\documentclass[11pt,reqno,tbtags]{amsart}
\usepackage{amssymb}
\pdfoutput=1
\usepackage[colorlinks=true, pdfstartview=FitV, linkcolor=blue,
  citecolor=blue, urlcolor=blue,pagebackref=false]{hyperref}
\usepackage[numeric,initials,nobysame]{amsrefs}
\usepackage{amsmath,amssymb,amsfonts,dsfont}
\usepackage{enumerate,graphicx}

\numberwithin{equation}{section}

\newtheorem{maintheorem}{Theorem}

\newtheorem{theorem}{Theorem}[section]
\newtheorem*{theorem*}{Theorem}
\newtheorem{lemma}[theorem]{Lemma}

\newtheorem{proposition}[theorem]{Proposition}
\newtheorem{corollary}[theorem]{Corollary}

\theoremstyle{definition}{

}

\theoremstyle{remark}{

\newtheorem*{remark*}{Remark}

}

\newenvironment{enumeratei}{\begin{enumerate}[\upshape (i)]}
                           {\end{enumerate}}

\newcommand{\R}{\mathbb R}

\newcommand{\E}{\mathbf{E}}
\renewcommand{\P}{\mathbf{P}}
\DeclareMathOperator{\var}{Var}

\renewcommand{\epsilon}{\varepsilon}
\newcommand{\tS}{\tilde{S}}

\newcommand{\given}{\, \big| \,}

\newcommand{\tX}{\tilde{X}}

\newcommand{\taumag}{\tau_{{\rm mag}}}

\newcommand{\one}{\boldsymbol{1}}

\newcommand{\deq}{:=}

\newcommand{\tmix}{t_\textsc{mix}}
\newcommand{\trel}{t_\textsc{rel}}
\newcommand{\texp}{t_\mathrm{exp}}

\newcommand{\F}{\mathcal{F}}

\newcommand{\temp}{\delta}
\newcommand{\gap}{\text{\tt{gap}}}
\newcommand{\magspace}{\mathcal{X}}

\begin{document}
\title[Mixing time evolution of Glauber dynamics]{The mixing time evolution of Glauber
dynamics for the mean-field Ising model}
%\title{The mixing time evolution of Glauber
%dynamics for the Mean-field Ising Model}
%\date{\today}
\date{}

\author{Jian Ding, \thinspace Eyal Lubetzky and Yuval Peres}

%\author{Jian Ding}
\address{Jian Ding\hfill\break
Department of Statistics\\
UC Berkeley\\
Berkeley, CA 94720, USA.}
\email{jding@stat.berkeley.edu}
\urladdr{}

%\author{Eyal Lubetzky}
\address{Eyal Lubetzky\hfill\break
Microsoft Research\\
One Microsoft Way\\
Redmond, WA 98052-6399, USA.}
\email{eyal@microsoft.com}
\urladdr{}

%\author{Yuval Peres}
\address{Yuval Peres\hfill\break
Microsoft Research\\
One Microsoft Way\\
Redmond, WA 98052-6399, USA.}
%, and\hfill\break
%Department of Mathematics\\
%UC Berkeley\\
%Berkeley, CA 94720, USA.}
\email{peres@microsoft.com}
\urladdr{}
\thanks{Research of J. Ding and Y. Peres was supported in part by NSF grant DMS-0605166.}

\begin{abstract}
We consider Glauber dynamics for the Ising model on the complete graph on $n$ vertices, known as the Curie-Weiss model. It is well-known that the mixing-time in the high temperature regime ($\beta < 1$) has order $n\log n$, whereas the mixing-time in the case $\beta > 1$ is exponential in $n$. Recently, Levin, Luczak and Peres proved that for any fixed $\beta < 1$ there is cutoff at time $\frac{1}{2(1-\beta)}n\log n$ with a window of order $n$, whereas the mixing-time at the critical temperature $\beta=1$ is $\Theta(n^{3/2})$. It is natural to ask how the mixing-time transitions from $\Theta(n\log n)$ to $\Theta(n^{3/2})$ and finally to $\exp\left(\Theta(n)\right)$. That is, how does the mixing-time behave when $\beta=\beta(n)$ is allowed to tend to $1$ as $n\to\infty$.

In this work, we obtain a complete characterization of the mixing-time of
the dynamics as a function of the temperature, as it approaches its critical point $\beta_c=1$. In particular, we find a scaling window of order $1/\sqrt{n}$ around the critical temperature.
In the high temperature regime, $\beta = 1 - \temp$ for some $0 < \temp < 1$ so that $\temp^2 n \to\infty$ with $n$, the mixing-time has order $(n/\temp)\log(\temp^2 n)$, and exhibits cutoff with constant $\frac{1}{2}$ and window size $n/\temp$.
 In the critical window, $\beta = 1\pm \temp$ where $\temp^2 n$ is $O(1)$, there is no cutoff, and the mixing-time has order $n^{3/2}$.
 At low temperature,
$\beta = 1 + \temp$ for $\temp > 0$ with $\temp^2 n \to\infty$
and $\delta=o(1)$,
there is no cutoff, and the mixing time has order
%$\frac{n}{\temp}\exp\left(\frac{n}{2}\int_{0}^{\zeta} \log
%\left(\frac{1+g(x)}{1-g(x)}\right)dx \right)$, where $g(x)=\frac{\tanh(\beta x)-x}{1-x\tanh(\beta x)}$, and
%$\zeta$ is the unique positive root of $g$. In particular, if $\delta=o(1)$ (low temperature near the scaling window), the mixing time has order
$\frac{n}{\delta}\exp\left((\frac{3}{4}+o(1))\delta^2 n\right)$.
%Furthermore, for each temperature, we determine the order of the spectral gap of the Glauber dynamics.
%Our results confirm a conjecture of the third author for the special case of the mean-field setting.
\end{abstract}

\maketitle

\section{Introduction}\label{sec:intro}

\begin{figure}
\centering \fbox{\includegraphics{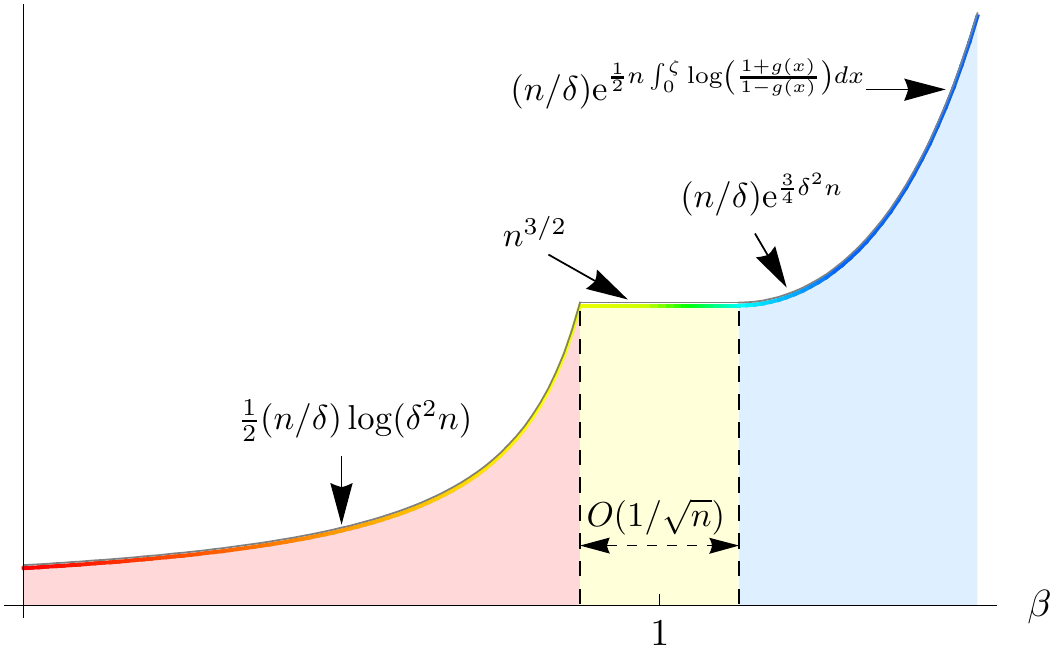}}
\caption{Illustration of the mixing time evolution as a function
of the inverse-temperature $\beta$, with a scaling window of order $1/\sqrt{n}$ around the critical point. We write $\delta = |\beta-1|$ and let $\zeta$ be the unique positive root of $g(x)\deq\frac{\tanh(\beta x)-x}{1-x\tanh(\beta x)}$.
Cutoff only occurs at high temperature.}
\label{fig:tmix-evol}
\end{figure}

The \emph{Ising Model} on a finite graph $G=(V,E)$ with parameter $\beta \geq 0$ and no external magnetic field is defined as follows. Its set of possible \emph{configurations} is $\Omega = \{1,-1\}^V$, where each configuration $\sigma\in\Omega$ assigns positive or negatives \emph{spins} to the vertices of the graph. The probability that the system is at a given configuration $\sigma$ is given by the \emph{Gibbs distribution}
$$\mu_G(\sigma) = \frac{1}{Z(\beta)} \exp\Big(\beta \sum_{xy\in E}\sigma(x)\sigma(y)\Big)~,$$
where $Z(\beta)$ (the partition function) serves as a normalizing constant. The parameter $\beta$ represents the inverse temperature: the higher $\beta$ is (the lower the temperature is), the more $\mu_G$ favors configurations where neighboring spins are aligned. At the extreme case $\beta = 0$ (infinite temperature), the spins are totally independent and $\mu_G$ is uniform over $\Omega$.

The \emph{Curie-Weiss} model corresponds to the case where the underlying geometry is the complete graph on $n$ vertices. The study of this model (see, e.g., \cite{Ellis},\cite{EN},\cite{ENR},\cite{LLP}) is motivated by the fact that its behavior approximates that of the Ising model on high-dimensional tori. It is convenient in this case to re-scale the parameter $\beta$, so that the stationary measure $\mu_n$ satisfies
\begin{equation}\label{eq-mu(sigma)}
\mu_n(\sigma) \propto \exp\Big(\frac{\beta}{n} \sum_{x < y} \sigma(x)\sigma(y)\Big)~.\end{equation}

The \emph{heat-bath Glauber dynamics} for the distribution $\mu_n$ is the following Markov Chain, denoted by $(X_t)$. Its state space is $\Omega$, and at each step, a vertex $x \in V$ is chosen uniformly at random, and its spin is updated as follows. The new spin of $x$ is randomly chosen according to $\mu_n$ conditioned on the spins of all the other vertices. It can  easily be shown that $(X_t)$ is an aperiodic irreducible chain, which is reversible with respect to the stationary distribution $\mu_n$.

We require several definitions in order to describe the mixing-time of the chain $(X_t)$. For any two
distributions $\phi,\psi$ on $\Omega$, the \emph{total-variation distance} of $\phi$ and $\psi$ is defined to be
$$\|\phi-\psi\|_\mathrm{TV} \deq \sup_{A \subset\Omega} \left|\phi(A) - \psi(A)\right| = \frac{1}{2}\sum_{\sigma\in\Omega} |\phi(\sigma)-\psi(\sigma)|~.$$
The (worst-case) total-variation distance of $(X_t)$ to stationarity at time $t$ is
$$ d_n(t) \deq \max_{\sigma \in \Omega} \| \P_\sigma(X_t \in \cdot)- \mu_n\|_\mathrm{TV}~,$$
where $\P_\sigma$ denotes the probability given that $X_0=\sigma$.
 The total-variation \emph{mixing-time} of $(X_t)$, denoted by $\tmix(\epsilon)$ for $0 < \epsilon < 1$, is defined to be
$$ \tmix(\epsilon) \deq \min\left\{t : d_n(t) \leq \epsilon \right\}~.$$
A related notion is the spectral-gap of the chain, $\gap \deq 1-\lambda$, where $\lambda$ is the largest absolute-value of all nontrivial eigenvalues of the transition kernel.

Consider an infinite family of chains $(X_t^{(n)})$, each with its corresponding worst-distance from stationarity $d_n(t)$, its mixing-times $\tmix^{(n)}$, etc. We say that $(X_t^{(n)})$ exhibits \emph{cutoff} iff for some sequence $w_n = o\big(\tmix^{(n)}(\frac{1}{4})\big)$ we have the following: for any $0 < \epsilon < 1$ there exists some $c_\epsilon > 0$, such that
\begin{equation}\label{eq-cutoff-def}\tmix^{(n)}(\epsilon) - \tmix^{(n)}(1-\epsilon) \leq c_\epsilon w_n \quad\mbox{ for all $n$}~.\end{equation}
That is, there is a sharp transition in the convergence of the given chains to equilibrium at time $(1+o(1))\tmix^{(n)}(\frac{1}{4})$. In this case, the sequence $w_n$ is called a \emph{cutoff window}, and the sequence $\tmix^{(n)}(\frac{1}{4})$ is called a \emph{cutoff point}.

It is well known that for any fixed $\beta>1$, the Glauber dynamics $(X_t)$ mixes in exponential time (cf., e.g., \cite{GWL}), whereas for any fixed $\beta < 1$ (high temperature) the mixing time has order $n\log n$ (see \cite{AH} and also \cite{BD}).
Recently, Levin, Luczak and Peres \cite{LLP} established that the mixing-time at the critical point $\beta=1$ has order $n^{3/2}$, and that for fixed $0<\beta<1$ there is cutoff at time $\frac{1}{2(1-\beta)}n\log n$ with window $n$. It is therefore natural to ask how the phase transition between these states occurs around the critical $\beta_c=1$: abrupt mixing at time $(\frac{1}{2(1-\beta)}+o(1)) n\log n$ changes to a mixing-time of $\Theta(n^{3/2})$ steps, and finally to exponentially slow mixing.

In this work, we determine this phase transition, and characterize the mixing-time of the dynamics as
a function of the parameter $\beta$, as it approaches its critical value $\beta_c=1$ both from below and from above. The scaling window around the critical temperature $\beta_c$ has order $1/\sqrt{n}$, as formulated by the following theorems, and illustrated in Figure \ref{fig:tmix-evol}.

\begin{maintheorem}[Subcritical regime]\label{thm-high-temp}
  Let $\temp =\delta(n) > 0$ be such that $\temp^2 n\to\infty$ with $n$.
  The Glauber dynamics for the mean-field Ising model with parameter $\beta=1-\temp$ exhibits cutoff at time $\frac{1}{2}(n/\temp)\log(\temp^2 n)$
  with window size $n/\temp$. In addition, the spectral gap of the dynamics in this regime is
 $(1+o(1))\temp/n$, where the $o(1)$-term tends to $0$ as $n\to\infty$.
\end{maintheorem}

\begin{maintheorem}[Critical window]\label{thm-critical-temp}
  Let $\temp =\temp(n)$ satisfy $\temp= O(1/\sqrt{n})$. The mixing time of the Glauber dynamics for the mean-field Ising model with parameter $\beta=1\pm\temp$ has order $n^{3/2}$, and does not exhibit cutoff.
    In addition, the spectral gap of the dynamics in this regime has order $n^{-3/2}$.
\end{maintheorem}

\begin{maintheorem}[Supercritical regime]\label{thm-low-temp}
Let $\temp =\temp(n)>0$ be such that $\temp^2 n\to\infty$ with $n$.
  The mixing-time of the Glauber dynamics for the mean-field Ising model
   with parameter $\beta=1+\temp$ does not exhibit cutoff, and has order $$\texp(n) \deq \frac{n}{\temp}\exp\left(\frac{n}{2}\int_{0}^{\zeta} \log \left(\frac{1+g(x)}{1-g(x)}\right)dx \right)~,$$
  where $g(x)\deq\left(\tanh(\beta x)-x\right)/\left(1-x\tanh(\beta x)\right)$, and $\zeta$ is the unique positive root of $g$.
  In particular, in the special case $\temp\to 0$, the order of the mixing time is $ \frac{n}{\temp}\exp \left((\frac{3}{4}+o(1))\temp^2 n\right)$, where the $o(1)$-term tends to $0$ as $n\to\infty$.
  In addition, the spectral gap of the dynamics in this regime has order $1/\texp(n)$.
\end{maintheorem}

%Our results show that cutoff occurs in the Curie-Weiss model only at the high-temperature regime, $\beta = 1-\delta$ for $0<\delta<1$ such that $\delta^2 n\to\infty$ with $n$. This confirms the following conjecture of the third author for the special case where the underlying graph is the complete graph.
%\begin{mainconj}
%Let $(G_{n})$ be a sequence of transitive graphs. If the Glauber dynamics on $(G_{n})$ has mixing-time $\tmix^{(n)}(\frac{1}{4}) = O(n\log n)$, then there is cutoff.
%\end{mainconj}

As we further explain in Section \ref{sec:outline}, the key element in the proofs of the above theorems is understanding the behavior of the sum of all spins (known as the \emph{magnetization chain}) at different temperatures. This function of the dynamics turns out to be an ergodic Markov chain as well, and namely a \emph{birth-and-death} chain (a one-dimensional chain, where only moves between neighboring positions are permitted). In fact, the reason for the exponential mixing at low-temperature is essentially that this
magnetization chain has two centers of mass, $\pm \zeta n$ (where $\zeta$ is as defined in Theorem \ref{thm-low-temp}), with an exponential commute time between them. Figure \ref{fig:mag-stat} demonstrates how the single center of mass around $0$ that this chain (rescaled) has at high near-critical temperature proceeds to split into two symmetric centers of mass that drift further and further apart as the temperature decreases.
\begin{figure}
\centering \fbox{\includegraphics[width=4in]{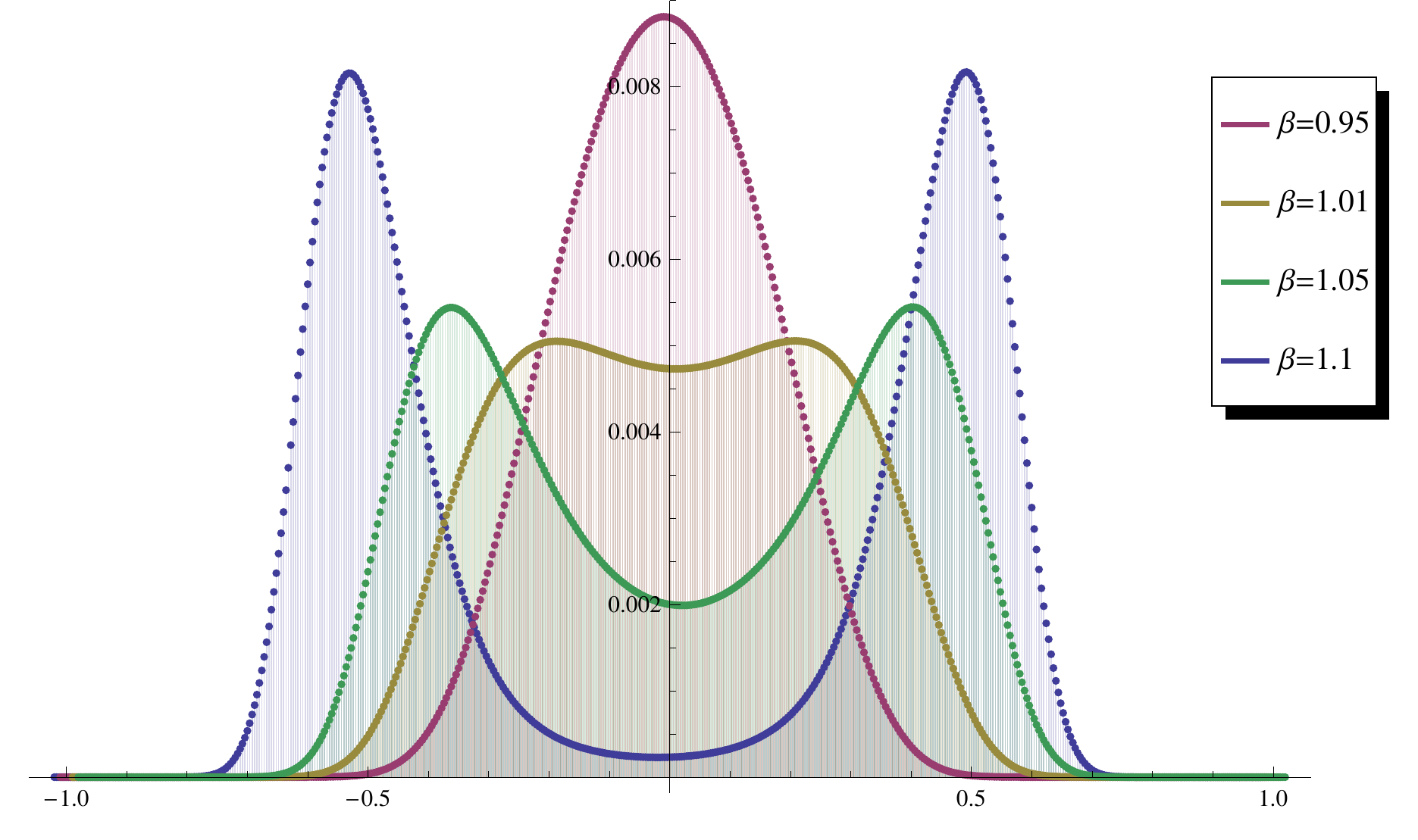}}
\caption{The stationary distribution of the normalized magnetization chain (average of all spins) for the dynamics on $n=500$ vertices. The center of mass at high temperatures (see $\beta=0.95$) is at $0$. Low temperatures feature two centers of mass at $\pm \zeta$ (where $\zeta$ is the unique
positive solution of $\tanh(\beta x) = x$), leading to the exponential mixing time. }
\label{fig:mag-stat}
\end{figure}

In light of this, a natural question that rises is whether the above mentioned bottleneck between the two centers of mass at $\pm \zeta n$ is the \emph{only} reason for the exponential mixing-time at low temperatures.
Indeed, as shown in \cite{LLP} for the strictly supercritical regime, $\beta >1$ fixed, if one restricts the Glauber dynamics to non-negative magnetization (known as the \emph{censored} dynamics), the mixing-time becomes $\Theta(n\log n)$ just like in the subcritical regime. Formally, the censored dynamics is defined as follows: at each step, a new state $\sigma$ is generated according to the original rule of the Glauber dynamics, and if a negative magnetization is reached ($S(\sigma) < 0$) then $\sigma$ is replaced by $-\sigma$. Interestingly, this simple modification suffices to boost the mixing-time back to order $n\log n$, just as in the high temperature case, and thus raises the question of whether the symmetry between the high temperature regime and the low temperature censored regime applies also to the existence of cutoff.

In a companion paper \cite{DLP-cens}, we strengthen the result of \cite{LLP} by showing that the scaling window of $1/\sqrt{n}$ exists also for the censored low temperature case, beyond which cutoff indeed occurs (yet at a different location than in the symmetric high temperature point).

\begin{maintheorem}\label{thm-cens-low-temp}
  Let $\temp > 0$ be such that $\temp^2 n\to\infty$ arbitrarily slowly with $n$. Then
  the \emph{censored} Glauber dynamics for the mean field Ising model
   with parameter $\beta=1+\temp$
    has a cutoff at $$t_n = \left(\frac{1}{2}+\frac{1}{2(\zeta^2 \beta/\temp - 1)}\right)\frac{n}{\temp}\log(\temp^2 n)$$
  with a window of order $n/\temp$. In the special case of the dynamics started from the all-plus configuration,
  the cutoff constant is $[2(\zeta^2 \beta/\temp - 1)]^{-1}$ (the order of the cutoff point and the window size remain the same).
\end{maintheorem}

\begin{maintheorem}\label{thm-cens-low-temp-spectral}
  Let $\temp>0$ be such that $\temp^2 n\to\infty$ arbitrarily slowly with $n$.
  Then the \emph{censored} Glauber dynamics for the mean field Ising model
   with parameter $\beta=1+\temp$
    has a spectral gap of order $\temp /n$.
\end{maintheorem}

%Namely, if $\beta = 1+\temp$ for some $\temp(n) > 0$ such that $\temp^2 n\to\infty$, then the censored dynamics exhibits cutoff
%at time $\left(1+\frac{1}{1 - \temp / (\zeta^2 \beta)}\right)\frac{n}{2\temp}\log(\temp^2 n)$ with window $n/\temp$, where $\zeta$ is as defined in Theorem \ref{thm-low-temp}.

Recalling Theorem \ref{thm-high-temp}, the above confirms that there is a symmetric scaling window of order $1/\sqrt{n}$ around the critical temperature, beyond which there is cutoff both at high and at low temperatures, with the same order of mixing-time (yet with a different constant), cutoff window and spectral gap.

%\subsection{Organization}
The rest of this paper is organized as follows. Section \ref{sec:outline} contains a brief outline of the proofs of the main theorems. Several preliminary facts on the Curie-Weiss model and on one-dimensional chains appear in Section \ref{sec:prelim}. Sections \ref{sec:high}, \ref{sec:critical} and \ref{sec:low} address the high temperature regime (Theorem \ref{thm-high-temp}), critical temperature regime (Theorem \ref{thm-critical-temp}) and low temperature regime (Theorem \ref{thm-low-temp}) respectively.
%The final section, Section \ref{sec:conclusion}, is devoted to open problems and concluding remarks.

\section{Outline of proof}\label{sec:outline}
In what follows, we present a sketch of the main ideas and arguments used in the proofs of the main theorems.
We note that the analysis of the critical window relies on arguments similar to those used for the subcritical and supercritical regimes. Namely, to obtain the order of the mixing-time in Theorem \ref{thm-critical-temp} (critical window), we study the magnetization chain using the arguments that appear in the proof of Theorem \ref{thm-high-temp} (high temperature regime). It is then straightforward to show that the mixing-time of the entire Glauber dynamics has the very same order. In turn, the spectral-gap in the critical window is obtained using arguments similar to those used in the proof of Theorem \ref{thm-low-temp} (low temperature regime). In light of this, the following sketch will focus on the two non-critical temperature regimes.

\subsection{High temperature regime}
\subsubsection*{Upper bound for mixing}
As mentioned above, a key element in the proof is the analysis of the \emph{normalized magnetization} chain, $(S_t)$, which is the average spin in the system. That is, for a given configuration $\sigma$, we define $S(\sigma)$ to be $\frac{1}{n}\sum_i \sigma(i)$, and it is easy to verify that this function of the dynamics is an irreducible and aperiodic Markov chain. Clearly, a necessary condition for the mixing of the dynamics is the mixing of its magnetization, but interestingly, in our case the converse essentially holds as well. For instance, as we later explain, in the special case where the starting state is the all-plus configuration, by symmetry these two chains have precisely the same total variation distance from equilibrium at any given time.

In order to determine the behavior of the chain $(S_t)$, we first keep track of its expected value along
 the Glauber dynamics. To simplify the sketch of the argument, suppose that our starting configuration
 is somewhere near the all-plus configuration. In this case, one can show that $\E S_t$ is monotone
    decreasing in $t$, and drops to order $\sqrt{1/\temp n}$ precisely at the cutoff point. Moreover, if we allow the dynamics to perform another $\Theta(n/\temp)$ steps (our cutoff window), then the
 magnetization will hit $0$ (or $\frac{1}{n}$, depending on the parity of $n$) with probability arbitrarily close to $1$. At that point, we essentially achieve the mixing of the magnetization chain.

It remains to extend the mixing of the magnetization chain to the mixing of the entire Glauber dynamics.
Roughly, keeping in mind the above comment on the symmetric case of the all-plus starting configuration, one can apply a similar argument to an arbitrary starting configuration $\sigma$, by separately treating the set of spins which were initially positive and those which were initially negative.
Indeed, it was shown in \cite{LLP} that the
following holds for $\beta < 1$ fixed (strictly subcritical regime). After a ``burn-in'' period of order $n$ steps,
the magnetization typically becomes not too biased. Next, if one runs two instances of the dynamics, from two
such starting configurations (where the magnetization is not too biased), then
by the time it takes their magnetization chains to coalesce, the entire configurations become relatively similar.
This was established by a so-called \emph{Two Coordinate Chain} analysis, where the two coordinates correspond to
the current sum of spins along the set of sites which were initially either positive or negative respectively.

By extending the above Two Coordinate Chain Theorem to the case of $\beta = 1-\delta$ where $\delta=\delta(n)$ satisfies $\delta^2 n \to \infty$, and combining it with second moment arguments and some additional ideas, we were able to show that the above behavior holds throughout this mildly subcritical regime. The burn-in time required for the typical magnetization to become ``balanced'' now has order $n/\delta$, and so does the time it takes the full dynamics of two chains to coalesce once their magnetization chains have coalesced. Thus, these two periods are conveniently absorbed in our cutoff window, making the cutoff of the magnetization chain the dominant factor in the mixing of the entire Glauber dynamics.

\subsubsection*{Lower bound for mixing}
While the above mentioned Two Coordinate Chain analysis was required in order to show that the entire Glauber dynamics mixes fairly quickly once its magnetization chain reaches equilibrium, the converse is immediate. Thus, we will deduce the lower bound on the mixing time of the dynamics solely from its magnetization chain.

The upper bound in this regime relied on an analysis of the first and second moments of the magnetization chain, however this approach is too coarse to provide a precise lower bound for the cutoff. We therefore resort to establishing an upper bound on the \emph{third} moment of the magnetization chain,
using which we are able to fine-tune our analysis of how its first moment changes along time. Examining the state of the system order $n/\delta$ steps before the alleged cutoff point, using concentration inequalities, we show that the magnetization chain is typically substantially far from $0$. Recalling Figure \ref{fig:mag-stat}, this implies a lower bound on the total variation distance of the magnetization chain to stationarity, as required.

\subsubsection*{Spectral gap analysis} In the previous arguments, we stated that the magnetization chain essentially dominates the mixing-time of the entire dynamics. An even stronger statement holds for the spectral gap: the Glauber dynamics and its magnetization chain have precisely the same spectral gap, and it is in both cases attained by the second largest eigenvalue. We therefore turn to establish the spectral gap of $(S_t)$.

The lower bound follows directly from the contraction properties of the chain in this regime. To obtain a matching upper bound, we use the Dirichlet representation for the spectral gap, combined with
an appropriate bound on the \emph{fourth} moment of the magnetization chain.

\subsection{Low temperature regime}%\label{subsec:outline-low}
\subsubsection*{Exponential mixing}
As mentioned above, the exponential mixing in this regime follows directly from the
behavior of the magnetization chain, which has a bottleneck between $\pm\zeta$.
To show this, we analyze the effective resistance between these two centers of mass, and obtain the precise order of the commute time between them. Additional arguments show that the mixing time of the entire Glauber dynamics in this regime has the same order.

\subsubsection*{Spectral gap analysis}
In the above mentioned proof of the exponential mixing, we establish that the commute time of the magnetization chain between $0$ and $\zeta$ has the same order as the hitting time from $1$ to $0$. We can therefore apply a recent result of \cite{DLP} for general birth-and-death chains, which implies that in this case the inverse of the spectral-gap (known as the relaxation-time) and the mixing-time must have the same order.

\section{Preliminaries} \label{sec:prelim}

\subsection{The magnetization chain}\label{subsec:magnetization}
The normalized magnetization of a configuration $\sigma\in\Omega$, denoted by $S(\sigma)$, is defined as
\begin{equation*}
S(\sigma) \deq \frac{1}{n} \sum_{i=1}^{n} \sigma(i)~.
\end{equation*}
Suppose that the current state of the Glauber dynamics is $\sigma$, and that site $i$ has been selected to have its spin updated. By definition, the
probability of updating this site to a positive spin is given by $p^{+}\left(S (\sigma)- \sigma(i)/n\right)$, where
\begin{equation}\label{eq-def-p-+}
p^{+}(s)\deq\frac{\mathrm{e}^{\beta s}} {\mathrm{e}^{\beta s} + \mathrm{e}^{-\beta s}}=\frac{1 + \tanh(\beta s)}{2}~.
\end{equation}
Similarly, the probability of updating the spin of site $i$ to a negative one is given by $p^{-}\left(S (\sigma)- \sigma(i)/n\right)$,
where
\begin{equation}\label{eq-def-p--}
 p^{-}(s)\deq \frac{\mathrm{e}^{-\beta s}}{\mathrm{e}^{\beta s} + \mathrm{e}^{-\beta s}}= \frac{1-\tanh(\beta s)}{2}~.
\end{equation}
It follows that the (normalized) magnetization of the Glauber dynamics at each step is a Markov chain, $(S_t)$,
with the following transition kernel:
\begin{align}\label{eq-magnet-transit}
P_M(s, s')=
\begin{cases}
\frac{1+s}{2}p^{-}(s-n^{-1}) & \hbox{ if } s' = s-\frac{2}{n}, \\
\frac{1-s}{2}p^{+}(s+ n^{-1}) & \hbox{ if } s'= s+ \frac{2}{n},\\
1-\frac{1+s}{2}p^{-}(s-n^{-1})- \frac{1-s}{2}p^{+}(s+ n^{-1}) & \hbox{ if } s'=s~.
\end{cases}
\end{align}
An immediate important property that the above reveals is the symmetry of $S_t$: the distribution
of $\left(S_{t+1} \mid S_t = s\right)$ is precisely that of $\left(-S_{t+1} \mid S_t = -s\right)$.

As evident from the above transition rules, the behavior of the Hyperbolic tangent will be useful in many arguments. This is illustrated in the following simple calculation, showing that the minimum over the holding probabilities of the magnetization chain is nearly $\frac{1}{2}$. Indeed, since the derivative of $\tanh(x)$ is bounded away from $0$ and $1$ for all $x\in[0,\beta]$ and any $\beta = O(1)$, the Mean Value Theorem gives
\begin{equation}\label{eq-holding-prob}
\begin{array}{rcc}
P_M(s,s+\frac{2}{n}) &=& \frac{1-s}{4}\left(1+\tanh(\beta s)\right)+O(n^{-1})~,\\
P_M(s,s-\frac{2}{n}) &=& \frac{1+s}{4}\left(1-\tanh(\beta s)\right)+O(n^{-1})~,\\
P_M(s,s) &=& \frac{1}{2}\left(1+s\tanh(\beta s)\right) - O(n^{-1})~.
\end{array}
\end{equation}
Therefore, the holding probability in state $s$ is at least $\frac{1}{2}-O\big(\frac{1}{n}\big)$. In fact, since $\tanh(x)$ is monotone increasing, $P_M(s,s) \leq \frac{1}{2}+\frac{1}{2}s\tanh(\beta s)$ for all $s$, hence these probability are also bounded from above by $\frac{1}{2}(1+\tanh(\beta)) < 1$.

Using the above fact, the next lemma will provide an upper bound for the coalescence time of two magnetization chains, $S_t$ and $\tS_t$, in terms of the hitting time $\tau_0$, defined as $\tau_0 \deq \min\{ t : |S_t| \leq n^{-1}\}$.
\begin{lemma}\label{lem-taumag-tau-0}
  Let $(S_t)$ and $(\tS_t)$ denote two magnetization chains, started from two arbitrary states.
  Then for any $\epsilon > 0$ there exists some $c_\epsilon > 0$, such that the following holds: if $T > 0$
  satisfies $\P_1(\tau_0 \geq T) < \epsilon$ then
   $S_t$ and $\tS_t$ can be coupled in a way such that they coalesce within at most $c_\epsilon T$ steps with probability at least $1-\epsilon$.
\end{lemma}
\begin{proof}
Assume without loss of generality that $|\tS_0| < |S_0|$, and by symmetry, that $\sigma = |S_0| \geq 0$. Define
$$\tau \deq \min\left\{t : |S_t| \leq |\tS_t| + \mbox{$\frac{2}{n}$} \right\}~.$$
Recalling the definition of $\tau_0$, clearly we must have $\tau < \tau_0$. Next, since the holding probability of $S_t$ at any state $s$ is bounded away from $0$ and $1$ for large $n$ (by the discussion preceding the lemma), there clearly exists a constant $0 < b < 1$ such that
 $$ \P\left(S_{t+1}=\tS_{t+1} \given |S_t - \tS_t| \leq \mbox{$\frac{2}{n}$}\right) > b > 0$$
(for instance, one may choose $b = \frac{1}{10}\left(1-\tanh(\beta)\right)$ for a sufficiently large $n$). It therefore follows that $|S_{\tau+1}| = |\tS_{\tau+1}|$ with probability at least $b$.

Condition on this event. We claim that in this case, the coalescence of $(S_t)$ and $(\tS_t)$ (rather than just their absolute values) occurs at some $t
\leq \tau_0+1$ with probability at least $b$. The case $S_{\tau+1} = \tS_{\tau+1}$ is immediate, and it remains
to deal with the case $S_{\tau+1} = -\tS_{\tau+1}$. Let us couple $(S_t)$ and $(\tS_t)$ so that the property
$S_t = -\tS_t$ is maintained henceforth. Thus, at time $t=\tau_0$ we obtain $|S_t - \tS_t| = 2|S_t| \leq
\frac{2}{n}$, and with probability $b$ this yields $S_{t+1}=\tS_{t+1}$.

Clearly, our assumption on $T$ and the fact that $0 \leq\sigma \leq 1$ together give $$\P_\sigma(\tau_0 \geq T) \leq \P_1(\tau_0 \geq T) < \epsilon~.$$
Thus, with probability
 at least $(1-\epsilon) b^2$, the coalescence time of $(S_t)$ and $(\tS_t)$ is at most
 $T$. Repeating this experiment a sufficiently large number of times then completes the proof.
\end{proof}

In order to establish cutoff for the magnetization chain $(S_t)$, we will need to carefully track its moments along the Glauber dynamics. By definition (see \eqref{eq-magnet-transit}), the behavior of these moments is governed by the Hyperbolic tangent function, as demonstrated by the following useful form for the conditional expectation of $S_{t+1}$ given $S_t$ (see also \cite{LLP}*{(2.13)}).
\begin{align}\label{eq-conditional-exp-S}
\E\left[S_{t+1} \mid S_t =s\right] &= \big(s+\mbox{$\frac{2}{n}$}\big) P_M\big(s, s+\mbox{$\frac{2}{n}$}\big) + s P_M(s, s) + \big(s-\mbox{$\frac{2}{n}$}\big)P_M\big(s, s-\mbox{$\frac{2}{n}$}\big)\nonumber\\
&=(1-n^{-1})s +\varphi(s) - \psi(s)~,
\end{align}
where
\begin{align*}
  \varphi(s) = \varphi(s,\beta,n)      & \deq \frac{1}{2n}\left[ \tanh\left(\beta(s+n^{-1})\right)
                       + \tanh\left(\beta(s-n^{-1})\right) \right]~, \\
  \psi(s) = \psi(s,\beta,n) & \deq \frac{s}{2n}\left[ \tanh\left(\beta(s+n^{-1})\right)
                       - \tanh\left(\beta(s-n^{-1})\right) \right]~.
\end{align*}

\subsection{From magnetization equilibrium to full mixing }
The motivation for studying the magnetization chain is that its mixing essentially dominates the
full mixing of the Glauber dynamics. This is demonstrated by the next straightforward lemma (see also \cite{LLP}*{Lemma 3.4}), which shows that in
the special case where the starting point is the all-plus configuration, the mixing of the magnetization is \emph{precisely} equivalent to that of the entire dynamics.

\begin{lemma}\label{lem-all-plus-mag-full}
Let $(X_t)$ be an instance of the Glauber dynamics for the mean field Ising model starting from the all-plus
configuration, namely, $\sigma_0 = \mathbf{1}$, and let $S_t = S(X_t)$ be its magnetization chain. Then
\begin{equation}
  \label{eq-allplus-xt-st-equiv}
  \| \P_{\mathbf{1}}(X_t \in \cdot)- \mu_n\|_\mathrm{TV} = \| \P_{\mathbf{1}}(S_t \in \cdot)- \pi_n\|_\mathrm{TV}~,
\end{equation}
where $\pi_n$ is the stationary distribution of the magnetization chain.
\end{lemma}
\begin{proof}
For any $s \in \{-1,-1+\frac{2}{n},\ldots,1-\frac{2}{n},1\}$, let $\Omega_s \deq \{\sigma \in \Omega: S(\sigma) = s\}$. Since by symmetry, both $\mu_n(\cdot \mid \Omega_s)$
and $\P_{\mathbf{1}}(X_t \in \cdot \mid S_t = s)$ are uniformly distributed over $\Omega_s$, the following holds:
\begin{align*}\|\P_{\mathbf{1}}(X_t\in\cdot)-\mu_n\|_\mathrm{TV} &= \frac{1}{2}\sum_{s} \sum_{\sigma\in\Omega_s} \left|\P_{\mathbf{1}} (X_t = \sigma) - \mu_n(\sigma)\right| \\
&= \frac{1}{2}\sum_s \sum_{\sigma\in\Omega_s} \Big|\frac{\P_\mathbf{1} (S_t = s)}{|\Omega_s|}
 - \frac{\mu_n(\Omega_s)}{|\Omega_s|}\Big|
\\
&= \|\P_\mathbf{1}(S_t\in\cdot)-\pi_n\|_\mathrm{TV}~.\qedhere
\end{align*}
\end{proof}
In the general case where the Glauber dynamics starts from an arbitrary configuration $\sigma_0$, though the above equivalence \eqref{eq-allplus-xt-st-equiv} no longer holds, the magnetization still dominates the full mixing of the dynamics in the following sense. The full coalescence of two instances of the dynamics occurs within order $n\log n$
steps once the magnetization chains have coalesced.

\begin{lemma}[\cite{LLP}*{Lemma 2.9}]\label{lem-fullmix-taumag}
 Let $\sigma,\tilde{\sigma} \in \Omega$ be such that
  $S(\sigma ) = S(\tilde{\sigma})$. For a coupling $(X_t,\tX_t)$, define the coupling time $\tau_{X,\tX} \deq \min\{t \geq 0: X_t=\tX_t\}$. Then for a sufficiently large $c_0 > 0$ there exists
  a coupling $(X_t, \tilde{X}_t)$ of the Glauber dynamics
  with initial states $X_0 = \sigma$ and $\tilde{X}_0 =
  \tilde{\sigma}$ such that
  \begin{equation*}
    \limsup_{n \rightarrow \infty}
     \P_{\sigma, \tilde{\sigma}}\left( \tau_{X,\tX} > c_0 n \log n \right)
       = 0~.
  \end{equation*}
\end{lemma}
Though Lemma \ref{lem-fullmix-taumag} holds for any temperature, it will only prove useful in the critical and
low temperature regimes. At high temperature, using more delicate arguments, we will establish full mixing
within order of $\frac{n}{\temp}$ steps once the magnetization chains have coalesced. That is, the extra steps required to achieve full mixing, once the magnetization chain cutoff had occurred, are absorbed in the cutoff window. Thus, in this regime, the entire dynamics has cutoff precisely when its magnetization chain does (with the same window).

\subsection{Contraction and one-dimensional Markov chains}
We say that a Markov chain, assuming values in $\R$, is \emph{contracting},
if the expected distance between two chains after a single step
decreases by some factor bounded away from $0$.
As we later show, the magnetization chain is contracting at high temperatures,
a fact which will have several useful consequences. One example of this is
the following straightforward lemma of \cite{LLP},
which provides a bound on the variance of the chain. Here and throughout the paper, the notation $\P_z$, $\E_z$ and $\var_z$
will denote the probability, expectation and variance respectively given that the starting state is $z$.
\begin{lemma}[\cite{LLP}*{Lemma 2.6}]\label{lem-var-bound}
  Let $(Z_t)$ be a Markov chain taking values in $\R$ and
  with transition matrix $P$. Suppose that there is some $0 < \rho < 1$ such
  that for all pairs of starting states $(z, \tilde{z})$,
  \begin{equation} \label{eq.abscontr}
    \left| \, \E_z[Z_t] - \E_{\tilde{z}}[Z_t] \, \right|
    \leq \rho^t |z - \tilde{z}| .
  \end{equation}
  Then $v_t \deq \sup_{z_0} \var_{z_0} (Z_t)$ satisfies $ v_t \leq  v_1 \min\left\{t, 1/\left(1-\rho^2\right)\right\}$.
\end{lemma}
\begin{remark*}
By following the original proof of the above lemma, one can readily extend it to the case $\rho \geq 1$ and get the following bound:
\begin{equation}\label{eq-var-bound-rho-geq-1}
    v_t \leq  v_1 \cdot \rho^{2t} \min\left\{t, 1/\left(\rho^2-1\right)\right\}~.
\end{equation}
This bound will prove to be effective for reasonably small values of $t$ in the critical window, where although the magnetization chain is not contracting, $\rho$ is only slightly larger than $1$.
\end{remark*}

Another useful property of the magnetization chain in the high temperature regime is its drift towards $0$. As
we later show, in this regime, for any $s > 0$ we have $\E \left[S_{t+1} | S_t = s\right] < s$, and with
probability bounded below by a constant we have $S_{t+1} < S_t$. We thus refer to the following lemma of
\cite{LPW}:
\begin{lemma}[\cite{LPW}*{Chapter 18}]\label{lem-supermatingale-positive}
  Let $(W_t)_{t \geq 0}$ be a non-negative
  supermartingale and $\tau$ be a stopping time such
  \begin{enumeratei}
    \item $W_0 = k$,
    \item $W_{t+1} - W_t \leq B$,
    \item $\var(W_{t+1} \mid \F_t) > \sigma^2 > 0$
      on the event $\tau > t$ .
  \end{enumeratei}
  If $u > 4 B^2/(3 \sigma^2)$, then $\P_k( \tau > u) \leq \frac{ 4 k }{\sigma \sqrt{u}}$.
\end{lemma}
This lemma, together with the above mentioned properties of $(S_t)$, yields the following immediate corollary:
\begin{corollary}[\cite{LLP}*{Lemma 2.5}]\label{cor-St-hitting-0}
  Let $\beta \le 1$, and suppose that $n$ is even.
  There exists a constant $c$ such that, for all $s$ and for all
  $u,t \ge 0$,
  \begin{equation} \label{eq.martrp}
    \P( \, |S_{u}| > 0, \ldots, |S_{u+t}| > 0 \mid S_u=s )
      \leq \frac{c n |s|}{\sqrt{t}}~.
  \end{equation}
\end{corollary}

Finally, our analysis of the spectral gap of the magnetization chain will require several results concerning birth-and-death chains from \cite{DLP}. In what follows and throughout the paper, the relaxation-time of a chain, $\trel$, is defined to be $\gap^{-1}$, where $\gap$ denotes its spectral-gap. We say that a chain  is $b$-\emph{lazy} if all its holding probabilities are at least $b$, or simply \emph{lazy} for the useful case of $b=\frac{1}{2}$. Finally, given an ergodic birth-and-death chain on $\magspace=\{0,1,\ldots,n\}$ with stationary distribution $\pi$, the \emph{quantile state} $Q(\alpha)$, for $0 < \alpha < 1$, is defined to be the smallest $i\in \magspace$ such
that $\pi(\{0,\ldots,i\}) \geq \alpha$.

\begin{lemma}[\cite{DLP}*{Lemma 2.9}]\label{lem-hitting-ratio}
Let $X(t)$ be a lazy irreducible birth-and-death chain on $\{0,1,\ldots,n\}$, and suppose that for some $0 < \epsilon <
\frac{1}{16}$ we have $\trel < \epsilon^4 \cdot \E_0 \tau_{Q(1-\epsilon)}$. Then for any fixed $\epsilon \leq
\alpha < \beta \leq 1-\epsilon$:
\begin{equation}\label{eq-commute-time}
  \E_{Q(\alpha)} \tau_{Q(\beta)} \leq \frac{3}{2\epsilon} \sqrt{ \trel \cdot \E_0 \tau_{Q(\frac{1}{2})}}~.
\end{equation}
\end{lemma}
\begin{lemma}[\cite{DLP}*{Lemma 2.3}]\label{lem-tv-hitting-bound}
For any fixed $0 < \epsilon < 1$ and lazy irreducible birth-and-death chain $X$, the following holds for any
$t$:
\begin{align}
  \label{eq-tv-from-0-hitting-bound}
  \| P^t(0,\cdot) - \pi \|_{\mathrm{TV}} &\leq \P_0(\tau_{Q(1-\epsilon)} > t) + \epsilon~,
\end{align}
and for all $k \in \Omega$,
\begin{align} \label{eq-tv-from-k-hitting-bound}
\| P^t(k,\cdot) - \pi \|_{\mathrm{TV}} &\leq \P_k(\max\{\tau_{Q(\epsilon)},\tau_{Q(1-\epsilon)}\} > t) +
2\epsilon~.
\end{align}
\end{lemma}
\begin{remark*}
As argued in \cite{DLP} (see Theorem 3.1 and its proof), the above two lemmas also hold for the case where the birth-and-death chain is not lazy but rather $b$-lazy for some constant $b > 0$. The formulation for this more general case incurs a cost of a slightly different constant in \eqref{eq-commute-time}, and replacing $t$ with $t/C$ (for some constant $C$) in
\eqref{eq-tv-from-0-hitting-bound} and \eqref{eq-tv-from-k-hitting-bound}. As we already established (recall
\eqref{eq-holding-prob}),
the magnetization chain is indeed $b$-lazy for any constant $b < \frac{1}{2}$ and a sufficiently large $n$. \end{remark*}

\subsection{Monotone coupling}\label{sec:monotone coupling} A useful tool throughout our arguments is the \emph{monotone coupling} of two instances of the Glauber dynamics $(X_t)$ and $(\tX_t)$, which maintains a coordinate-wise inequality between the corresponding configurations. That is, given two configurations $\sigma \geq \tilde{\sigma}$ (i.e., $\sigma(i)\geq \tilde{\sigma}(i)$ for all $i$), it is possible to generate the next two states $\sigma'$ and $\tilde{\sigma}'$ by updating the same site in both, in a manner that ensures that $\sigma' \geq \tilde{\sigma}'$. More precisely, we draw a random variable $I$ uniformly over
$\{1, 2, \ldots, n\}$ and independently draw another random variable $U$ uniformly over $[0,1]$. To generate
$\sigma'$ from $\sigma$, we update site $I$ to $+1$ if $U \leq p^{+}\left(S(\sigma)-\frac{\sigma(I)}{n}\right)$,
 otherwise $\sigma'(I)=-1$. We perform an analogous process in order to generate $\tilde{\sigma}'$ from
$\tilde{\sigma}$, using the same $I$ and $U$ as before. The monotonicity of the function $p^{+}$ guarantees that $\sigma' \geq \tilde{\sigma}'$, and by repeating this process, we obtain a coupling of the two instances of the
 Glauber dynamics that always maintains monotonicity.

Clearly, the above coupling induces a monotone coupling for the two corresponding magnetization chains.
We say that a birth-and-death chain with a transition kernel $P$ and a state-space $\magspace=\{0,1,\ldots,n\}$
is \emph{monotone} if $P(i,i+1) + P(i+1,i) \leq 1$ for every $i < n$. It is easy to verify that this condition
is equivalent to the existence of a monotone coupling, and that for such a chain, if $f:\magspace\to\R$ is a monotone increasing (decreasing) function then so is $P f$ (see, e.g., \cite{DLP}*{Lemma 4.1}).

%In addition, at certain times we will also require a monotone coupling for the absolute magnetization chain, $|S_t|$.
%The symmetry of $S_t$ enables us to use the usual monotone coupling (as described above) for this purpose, everywhere except when the two chains $S_t$ and $\tS_t$ assume the values $0$ and $\pm\frac{2}{n}$ for even $n$. To verify that a monotone coupling exists also for this final special case, assume that $S_0 = \frac{2}{n}$ and $\tS_0= 0$; we need to construct a coupling such that $|S_1| \geq |\tS_1|$. Indeed, this is possible, as guaranteed by the following fact:
%\begin{equation*}
%P_M(0, \mbox{$\frac{2}{n}$}) + P_M (0, \mbox{$-\frac{2}{n}$}) + P_M (\mbox{$\frac{2}{n}$}, 0) \leq 1~.
%\end{equation*}

\subsection{The spectral gap of the dynamics and its magnetization chain} \label{subsec:spetral}
To analyze the spectral gap of the Glauber dynamics, we establish the following lemma which reduces this
problem to determining the spectral-gap of the one-dimensional magnetization chain. Its proof relies
on increasing eigenfunctions, following the ideas of \cite{Nacu}.
\begin{proposition}\label{prop-spetral-glauber-mag}
   The Glauber dynamics for the mean-field Ising model and its one-dimensional magnetization chain have the same spectral gap. Furthermore, both gaps are attained by the largest
nontrivial eigenvalue.
\end{proposition}
\begin{proof}
We will first show that
the one-dimensional magnetization chain has an increasing eigenfunction, corresponding to the second eigenvalue.

Recalling that $S_t$ assumes values in $\magspace\deq\{-1,-1+\frac{2}{n},\ldots,1-\frac{2}{n},1\}$, let $M$ denote its transition matrix, and let $\pi$ denote its stationary distribution. Let $1=
\theta_0 \geq \theta_1\geq \ldots \geq \theta_{n}$ be the $n+1$ eigenvalues of $M$,
corresponding to the eigenfunctions $f_0 \equiv 1, f_1, \ldots, f_n$. Define
$\theta = \max\{\theta_1,|\theta_n|\}$, and notice that, as $S_t$ is aperiodic and irreducible, $0 <\theta < 1$. Furthermore, by the existence of the monotone coupling for $S_t$ and the discussion in the previous subsection, whenever a function $f:\magspace\to\R$ is increasing so is $M f$.

Define $f:I\to\R$ by $f \deq f_1 +f_n + K \mathds{1}$, where $\mathds{1}$ is the identity function and $K>0$ is sufficiently large to ensure that $f$ is monotone increasing (e.g., $K = \frac{n}{2} \|f_1+f_n\|_{L^\infty}$ easily suffices). Notice that, by symmetry of $S_t$, $\pi(x) = \pi(-x)$ for all $x\in \magspace$, and in particular
$\sum_{x\in \magspace} x \pi(x) = 0$, that is to say, $\left<\mathds{1},f_0\right>_{L^2(\pi)} = 0$. Recalling that for all $i\neq j$ we have $\left<f_i,f_j\right>_{L^2(\pi)}=0$, it follows that for some $q_1,\ldots,q_n\in \R$ we have $f = \sum_{i=1}^n q_i f_i$ with $q_1\neq 0$ and $q_n \neq 0$,
and thus
\begin{equation*}
\left(\theta^{-1} M\right)^{m} f = \sum_{i=1}^{n} q_i (\theta_i / \theta)^{m} f_i~.
\end{equation*}
Next, define
\begin{equation*}
g=\left\{\begin{array}
  {ll} q_1 f_1 & \mbox{if }\theta = \theta_1\\
  0 & \mbox{otherwise}
\end{array}\right.
~,\quad\mbox{and}\quad
h=\left\{\begin{array}
  {ll} q_n f_n & \mbox{if }\theta = -\theta_n\\
  0 & \mbox{otherwise}
\end{array}\right.~,
\end{equation*}
and notice that
$$ \lim_{m\to\infty} \left(\theta^{-1} M\right)^{2m} f = g + h~,\quad\mbox{and}\quad\lim_{m\to\infty}\left(\theta^{-1} M\right)^{2m+1} f = g - h~.$$
As stated above, $M^m f$ is increasing for all
$m$, and thus so are the two limits $g+h$ and $g-h$ above, as well as their sum. We deduce that $g$ is an increasing function, and next claim that $g \not\equiv 0$. Indeed, if $g \equiv 0$ then both $h$ and $-h$ are increasing functions, hence necessarily $h \equiv 0$ as well; this would imply that $q_1=q_n = 0$, thus contradicting our construction of $f$.

We deduce that $g$ is an increasing eigenfunction corresponding to $\theta_1 = \theta$, and next wish to show that it is strictly increasing. Recall that for all $x \in \magspace$,
\begin{equation*}
  (M g)(x) = M\Big(x, x-\frac{2}{n}\Big) g\Big(x-\frac{2}{n}\Big) + M(x, x) g(x) + M\Big(x, x + \frac{2}{n}\Big) g\Big(x +\frac{2}{n}\Big)~.
\end{equation*}
Therefore, if for some $x\in \magspace$ we had $g(x-\frac{2}{n}) = g(x) \geq 0$, the fact that $g$ is increasing would imply that
$$ \theta_1 g(x) = (M g)(x) \geq  g(x) \geq 0~,$$
and analogously, if $g(x) = g(x+\frac{2}{n}) \leq 0$ we could write
$$ \theta_1 g(x) = (M g)(x) \leq  g(x) \leq 0~.$$
In either case, since $0 < \theta_1 < 1$ (recall that $\theta_1 = \theta$) this would in turn lead to $g(x) = 0$. By inductively substituting this fact in the above equation for $(Mg)(x)$, we would immediately get $g\equiv 0$, a contradiction.

Let $1 = \lambda_0 \geq \lambda_1 \geq \ldots \geq \lambda_{|\Omega|-1}$ denote the eigenvalues of the Glauber dynamics, and let $\lambda \deq \max\{\lambda_1, |\lambda_{2^n-1}|\}$. We translate $g$ into a function $G:\Omega\to \R$ in the obvious manner:
\begin{equation*}
G(\sigma) \deq g(S(\sigma))= g\Big({\frac{1}{n}\sum_{i=1}^{n}\sigma(i)}\Big)~.
\end{equation*}
One can verify that $G$ is indeed an eigenfunction of the Glauber dynamics corresponding to the eigenvalue
$\theta_1$, and clearly $G$ is strictly increasing with respect to the coordinate-wise partial order on $\Omega$. At this point, we refer to the following lemma of \cite{Nacu}:
\begin{lemma}[\cite{Nacu}*{Lemma 4}]
Let $P$ be the transition matrix of the Glauber dynamics, and let $\lambda_1$ be its second largest eigenvalue.
If $P$ has a strictly increasing eigenfunction $f$, then $f$ corresponds to $\lambda_1$.
\end{lemma}
The above lemma immediately implies that $G$ corresponds to the second eigenvalue of Glauber dynamics, which we denote by $\lambda_1$, and thus $\lambda_1 = \theta_1$.

It remains to show that $\lambda = \lambda_1$. To see this, first recall that all the holding probabilities of $S_t$  are bounded away from $0$, and the same applies to the entire Glauber dynamics by definition (the magnetization  remains the same if and only if the configuration remains the same). Therefore, both $\theta_n$ and $\lambda_{2^n-1}$ are bounded away from $-1$, and it remains to show that $\gap = o(1)$ for the Glauber dynamics (and hence also for its magnetization chain).

To see this, suppose $P$ is the transition kernel of the Glauber dynamics, and recall the Dirichlet representation for the second eigenvalue of a reversible chain (see \cite{LPW}*{Lemma 13.7}, and also \cite{AF}*{Chapter 3}):
\begin{align}\label{eq-dirichlet-form}
1 - \lambda_1 = \min \Big\{ \frac{\mathcal{E}(f)}
{\left<f,f\right>_{\mu_n}} \;:\; f\not\equiv 0\;,\;\E_{\mu_n}(f)=0 \Big\}~,
\end{align}
where $\E_{\mu_n}(f)$ denotes $\left<\mathds{1},f\right>_{\mu_n}$, and
$$\mathcal{E}(f) = \left<(I-P)f,f\right>_{\mu_n} = \frac{1}{2}\sum_{\sigma,\sigma'\in\Omega}\left[f(\sigma)-f(\sigma')\right]^2\mu_n(\sigma)P(\sigma,\sigma')~. $$
By considering the sum of spins, $h(\sigma) = \sum_{i=1}^n \sigma(i)$, we get $\mathcal{E}(h) \leq 2$, and since
the spins are positively correlated, $\var_{\mu_n} \sum_{i} \sigma(i) \geq n$. It follows that
\begin{equation*}
  1-\lambda_1 \leq 2/n~,
\end{equation*}
and thus $\gap = 1-\lambda_1 = 1-\theta_1$ for both the Glauber dynamics and its magnetization chain, as required.
\end{proof}

\section{High temperature regime}\label{sec:high}
In this section we prove Theorem \ref{thm-high-temp}. Subsection \ref{sec:high-cutoff-mag} establishes the cutoff of the magnetization chain, which immediately provides a lower bound on the mixing time of the entire dynamics. The matching upper bound, which completes the proof of cutoff for the Glauber dynamics, is given in Subsection \ref{sec:high-full-mixing}. The spectral gap analysis appears in Subsection \ref{sec:high-spectral}. Unless stated otherwise, assume throughout this section that $\beta = 1 - \temp$ where $\temp^2 n \to \infty$.

\subsection{Cutoff for the magnetization chain}\label{sec:high-cutoff-mag}
Clearly, the mixing of the Glauber dynamics ensures the mixing of its magnetization. Interestingly, the converse is also essentially true, as the mixing of the magnetization
turns out to be the most significant part in the mixing of the full Glauber dynamics.
We thus wish to prove the following cutoff result:
\begin{theorem}
  \label{thm-high-cutoff-mag}
Let $\beta = 1-\delta$, where $\delta > 0$ satisfies $\delta^2 n\to\infty$. Then the corresponding
magnetization chain $(S_t)$ exhibits cutoff at time $\frac{1}{2}\cdot\frac{n}{\temp}\log (\temp^2 n)$ with a window of order $n/\temp$.
\end{theorem}
Notice that Lemma \ref{lem-all-plus-mag-full} then gives the following corollary for the special case where the initial state of the dynamics is the all-plus configuration:
\begin{corollary}\label{cor-high-from-all-plus}
  Let $\temp =\delta(n) > 0$ be such that $\temp^2 n\to\infty$ with $n$, and let $(X_t)$ denote
    the Glauber dynamics for the mean-field Ising model with parameter $\beta=1-\temp$, started
    from the all-plus configuration. Then $(X_t)$ exhibits cutoff at time $\frac{1}{2}(n/\temp)\log(\temp^2 n)$
  with window size $n/\temp$.
\end{corollary}

\subsubsection{Upper bound}\label{sec:high-upper-mag}
Our goal in this subsection is to show the following:
\begin{align}\label{eq-sub-upperbound-mixing}
\lim_{\gamma \to\infty} \limsup_{n\to \infty} \; d_n \left(\frac{1}{2}\cdot\frac{n}{\temp}\log (\temp^2 n) + \gamma
\frac{n}{\temp}\right) = 0~,
\end{align}
where $d_n(\cdot)$ is with respect to the magnetization chain $(S_t)$ and its stationary distribution.
This will be obtained using an upper bound on the coalescence time of two instances of the magnetization chain.
Given the properties of its stationary distribution (see Figure \ref{fig:mag-stat}), we will mainly be interested in the time it takes this chain to hit near $0$.
The following theorem provides an upper bound for that hitting time.
\begin{theorem}\label{thm-tau-0}
For $0 < \beta < 1 + O(n^{-1/2})$, consider the magnetization chain started from some arbitrary state $s_0$, and let $\tau_0=\min\{ t : |S_t| \leq n^{-1}\}$. Write $\beta = 1-\temp$, and for $\gamma > 0$ define
  \begin{equation}\label{eq-tn-def}
t_n(\gamma) =\left\{\begin{array}
  {ll}
  \displaystyle{\frac{n}{2\temp}\log(\temp^2 n) + (\gamma+3)\frac{n}{\temp}} &  \temp^2 n \to\infty~,\\
  \noalign{\medskip}
  \displaystyle{\left(200 + 6\gamma\left(1+ 6\sqrt{\temp^2 n}\right)\right) n^{3/2}} & \temp^2 n = O(1)~.
\end{array}\right.\end{equation}
Then there exists some $c>0$ such that $\P_{s_0}( \tau_0 > t_n(\gamma)) \leq c / \sqrt{\gamma}~$.
\end{theorem}
\begin{proof}
For any $t \geq 1$, define:
$$ s_t \deq \E_{s_0}\left[ |S_t| \one_{\{\tau_0 > t\}}\right]~.$$
Suppose $s > 0$. Recalling \eqref{eq-conditional-exp-S} and bearing in mind the concavity of the Hyperbolic tangent and the fact that $\psi(s) \geq 0$, we obtain that
\begin{equation*}
\E (S_{t+1} \given S_t =s)\leq s + \frac{1}{n}\big(\tanh (\beta s) -s\big)~.
\end{equation*}
Using symmetry for the case $s<0$, we can then deduce that
\begin{equation} \label{eq.ineq-2.17}
  \E\left[ |S_{t+1}| \given S_t \right]
    \leq |S_t|
       + \frac{1}{n}\big( \tanh(\beta |S_t|) - |S_t| \big) \mbox{ for any $t < \tau_0$}~.
\end{equation}
Hence, combining the concavity of the Hyperbolic tangent together with Jensen's inequality yields
\begin{equation}\label{eq-st-one-step}
  s_{t+1} \leq \left(1-\frac{1}{n}\right)s_t + \frac{1}{n}\tanh(\beta s_t)~.
\end{equation}
Since the Taylor expansion of $\tanh(x)$ is
\begin{equation}\label{eq-tanh-taylor}
\tanh(x) = x - \frac{x^3}{3} + \frac{2x^5}{15} - \frac{17x^7}{315}+O(x^9)~,
\end{equation} we have
$\tanh(x) \leq x-\frac{x^3}{5}$ for $0 \leq x \leq 1$, giving
\begin{align}\label{eq-s+-decreasing}
s_{t+1} &\leq \Big(1-\frac{1}{n}\Big)s_t + \frac{1}{n} \tanh(\beta s_t)
\leq \Big(1-\frac{1}{n}\Big)s_t + \frac{1}{n}\beta s_t - \frac{(s_t)^3}{5n}  \nonumber\\
&= s_t - \frac{\temp}{n}s_t -\frac{(s_t)^3}{5n} ~.
\end{align}
For some $1 < a \leq 2$ to be defined later, set
$$b_i = a^{-i}/4~,\mbox{ and }u_i = \min\{t:s_t \leq b_i\}~.$$
Notice that $s_t$ is decreasing in $t$ by \eqref{eq-s+-decreasing}, thus for every $t \in [u_i,u_{i+1}]$ we have $$b_i/a =
b_{i+1} \leq s_t \leq b_i~.$$ It follows that
$$ s_{t+1} \leq s_t - \frac{\temp}{n}\cdot \frac{b_i}{a} - \frac{b_i^3}{5a^3 n}~,$$
and
\begin{equation}\label{eq-ui-iteration}u_{i+1}-u_i \leq \left(\frac{a-1}{a}b_i\right)/\left(\frac{\temp}{n}\frac{b_i}{a} + \frac{b_i^3}{5a^3n}\right)
\leq \frac{5(a-1)a^2 n}{5\temp a^2 + b_i^2}~.\end{equation}

For the case $\temp^2 n  \to \infty$, define:
\begin{equation*}i_0 = \min\{i: b_i \leq 1/\sqrt{\temp n}\}~.
\end{equation*}
The following holds:
\begin{align*} \sum_{i=1}^{i_0} (u_{i+1}-u_i) &\leq
\sum_{i=1}^{i_0} \frac{5(a-1)a^2 n}{5\temp a^2  + b_i^2} \leq \mathop{\sum_{ i\leq i_0}}_{b_i^2 > \temp}
\frac{5(a-1)a^2 n}{b_i^2} +
\mathop{\sum_{i\leq i_0}}_{b_i^2 \leq \temp} \frac{5(a-1)a^2 n}{5\temp a^2 } \\
&\leq \frac{5n a^2}{\temp(a+1)} + \frac{a-1}{2\log a}\cdot \frac{n}{\temp} \log(\temp^2 n)~,
\end{align*}
where in the last inequality we used the fact that the series $\{b_i^{-2}\}$ is a geometric series with a ratio
$a^2$, and that, as $b_i^2 \geq 1/(\temp n)$ for all $i \leq i_0$, the number of summands such that $b_i^2 \leq
\temp$ is at most $\log_a(\sqrt{\temp^2 n})$. Therefore, choosing $a = 1+n^{-1}$, we deduce that:
\begin{align}
\sum_{i=1}^{i_0} (u_{i+1}-u_i) &\leq \left(\frac{5}{2}+O(n^{-1})\right)\frac{n}{\temp} + \left(\frac{1}{2}+O(n^{-1})\right)\frac{n}{\temp}\log(\temp^2 n)\nonumber\\
&\leq 3\frac{n}{\temp} + \frac{n}{2\temp}\log(\temp^2 n)~,%\label{eq-u-i0-bound}
\end{align}
where the last inequality holds for any sufficiently large $n$. Combining the above inequality and the
definition of $i_0$, we deduce that
\begin{equation}\label{eq-Es0-tau0-bound}
(\;s_{t_n(0)} =\;)~ \E_{s_0}\left[ |S_{t_n(0)}| \one_{\{\tau_0 > t_n(0)\}}\right] \leq 1/\sqrt{\temp n}~.\end{equation}
Thus, by
Corollary \ref{cor-St-hitting-0} (after taking expectation), for some fixed $c>0$
$$\P( \tau_0 > t_n(\gamma) ) \leq c / \sqrt{\gamma}~.$$

For the case $\temp^2 n = O(1)$, choose $a=2$, that is, $b_i=2^{-(i+2)}$, and define
\begin{equation}\label{eq-bi-bounds}
  i_1 = \min\{i: b_i \leq n^{-1/4} \vee 5 \sqrt{|\temp|}\}~.
\end{equation}
Substituting $a=2$ in \eqref{eq-ui-iteration}, while noting that $\delta > -\frac{1}{25}b_i^2$ for all $i < i_1$, gives
\begin{align*} \sum_{i=1}^{i_1} (u_{i+1}-u_i) &\leq
\sum_{i=1}^{i_1} \frac{20 n}{20\temp + b_i^2} \leq
100 \sum_{i=1}^{i_1} \frac{n}{b_i^2} \leq 200 \frac{n}{b_{i_1}^2} \\
&\leq \left(200n^{3/2} \wedge 8 \frac{n}{|\delta|} \right)
\leq 200 n^{3/2}~.
\end{align*}
where the last inequality in the first line incorporated the geometric sum over $\{b_i^{-2}\}$.
By \eqref{eq-bi-bounds},
$$ b_{i_1} \leq n^{-1/4} \vee 5\sqrt{|\temp|} \leq n^{-1/4}\left(1+5(\temp^2 n)^{1/4}\right)~,$$
and as in the subcritical case, we now combine the above results with an application of Corollary \ref{cor-St-hitting-0} (after taking expectation), and deduce that for some
absolute constant $c> 0$,
$$\P( \tau_0 > t_n(\gamma) ) \leq c / \sqrt{\gamma}~, $$
as required.
\end{proof}

Apart from drifting toward 0, and as we had previously mentioned, the magnetization chain at high temperatures is in fact contracting; this is a
special case of the following lemma.

\begin{lemma}\label{lem-contracting}
Let $(S_t)$ and $(\tS_t)$ be the corresponding magnetization chains of two instances of the Glauber dynamics for some $\beta=1-\temp$ (where $\temp$ is not necessarily positive), and put $D_t \deq S_t - \tS_t$. The following then holds:
\begin{equation}\label{eq-contraction}
\E[ D_{t+1}-D_t \mid D_t ] \leq - \frac{\temp}{n}D_t + \frac{|D_t|}{n^2} + O(n^{-4})~.
\end{equation}
\end{lemma}
\begin{proof}
By definition (recall \eqref{eq-conditional-exp-S}), we have
\begin{align*}\E[ D_{t+1}-D_t \mid D_t] &= \E[ S_{t+1}-S_{t} + \tS_t - \tS_{t+1} \mid D_t] \\
&= \frac{\tS_t-S_t}{n} + \left[\varphi(S_t) - \varphi(\tS_t)\right] - \left[ \psi(S_t) - \psi(\tS_t)\right]~.
\end{align*}
The Mean Value Theorem implies that
\begin{align*}\varphi(S_t) - \varphi(\tS_t) \leq \frac{\beta}{n}(S_t - \tS_t)~,
\end{align*}
and applying Taylor expansions on $\tanh(x)$ around $\beta S_t$ and $\beta \tS_t$, we deduce that
\begin{align*}\psi(S_t) - \psi(\tS_t) =
 \frac{S_t}{n^2\cosh^2(\beta S_t)}  - \frac{\tS_t}{n^2\cosh^2(\beta \tS_t)} + O\Big(\frac{1}{n^4}\Big)~.
\end{align*}
Since the derivative of the function $x / \cosh^2(\beta x)$ is bounded by $1$, another application of the Mean
Value Theorem gives
\begin{align*}\left|\psi(S_t) - \psi(\tS_t) \right| \leq
\frac{|S_t-\tS_t|}{n^2} + O\Big(\frac{1}{n^4}\Big)~.
\end{align*}
Altogether, we obtain \eqref{eq-contraction}, as required.
%\begin{align*}
%\E[ D_{t+1}-D_t \mid D_t ] \leq - \frac{\temp}{n}D_t + \frac{|D_t|}{n^2} + O(n^{-4})~,
%\end{align*}
%as required. The special case of $\beta$ then follows by simple algebra.
\end{proof}
Indeed, the above lemma ensures that in the high temperature regime, $\beta=1-\temp$ where $\temp > 0$, the magnetization chain is contracting:
\begin{align}\label{eq-high-temp-contraction}
\E\left[|D_{t+1}|\given D_t\right] \leq \Big(1-\frac{\temp}{2n}\Big)|D_t|~\mbox{ for any sufficiently large $n$}~.
\end{align}
We are now ready to prove that hitting near $0$ essentially ensures the mixing of the magnetization.
\begin{lemma}
  \label{lem-mag-coalesce-from-0}
  Let $\beta = 1-\delta$ for $\delta >0$ with $\delta^2 n\to\infty$, $(X_t)$ and $(\tX_t)$ be two instances of the dynamics started from arbitrary states $\sigma_0$ and $\tilde{\sigma}_0$ respectively, and $(S_t)$ and $(\tS_t)$ be their corresponding magnetization chains. Let $\taumag$ denote the coalescence time
  $\taumag := \min\{t: S_t = \tS_t\}$, and $t_n(\gamma)$ be as defined in Theorem \ref{thm-tau-0}.
Then there exists some constant $c>0$ such that \begin{equation}\label{eq-mag-coalescene-bound}
  \P\left(\taumag > t_n(3\gamma)\right) \leq c/\sqrt{\gamma}~\mbox{ for all $\gamma > 0$}~.\end{equation}
\end{lemma}
\begin{proof}
Set $T = t_n(\gamma)$. We claim that the following holds for large $n$:
\begin{equation}\label{eq-Est-bound-from-T}
|\E S_t| \leq \frac{2}{\sqrt{\temp n}}~\mbox{ and }~ |\E \tS_t| \leq \frac{2}{\sqrt{\temp n}}~\mbox{ for all $t \geq T$}~.
\end{equation}
To see this, first consider the case where $n$ is even. The above inequality then follows directly from  \eqref{eq-Es0-tau0-bound} and the
decreasing property of $s_t$ (see \eqref{eq-s+-decreasing}), combined with the fact that $\E_0 S_t = 0$ (and thus $\E S_t = 0$ for all $t \geq \tau_0$). In fact, in case $n$ is even, $|\E S_t|$ and $|\E \tS_t|$ are both at most $1/\sqrt{\temp n}$ for all $t \geq T$. For the case of $n$ odd (where there is no $0$ state for the magnetization chain, and  $\tau_0$ is the
hitting time to $\pm\frac{1}{n}$), a simple way to show that \eqref{eq-Est-bound-from-T} holds is to bound $|\E_{\frac{1}{n}} S_t|$. By definition, $P_M(\frac{1}{n}, \frac{1}{n}) \geq P_M
(\frac{1}{n}, -\frac{1}{n})$ (see \eqref{eq-magnet-transit}). Combined with the symmetry of the positive and negative parts of the magnetization
chain, one can then verify by induction that $P_M^t(\frac{1}{n}, \frac{k}{n}) \geq P_M^{t}(\frac{1}{n}, -\frac{k}{n})$ for any
odd $k>0$ and any $t$. Therefore, by symmetry as well as the fact that
$\E_{s_0} S_t \leq s_0$ for positive $s_0$, we conclude that $|\E_{\frac{1}{n}} S_t|$ is decreasing with $t$, and
thus is bounded by $\frac{1}{n}$. This implies that \eqref{eq-Est-bound-from-T} holds for odd $n$ as well.

Combining \eqref{eq-Est-bound-from-T} with the Cauchy-Schwartz inequality we obtain that for any $t \geq T$
$$ \E |S_t - \tS_t| \leq \E|S_t| + \E|\tS_t| \leq \sqrt{\var(S_t) + \frac{4}{\temp n}} + \sqrt{\var(\tS_t) +\frac{4}{\temp n}}~.$$
Now, combining Lemma \ref{lem-var-bound} and Lemma \ref{lem-contracting} (and in particular, \eqref{eq-high-temp-contraction}), we deduce that $\var{S_t} \leq
\frac{4}{\temp n}$, and plugging this into the above inequality gives
\begin{equation*}
  \E|S_t - \tS_t| \leq \frac{10}{\sqrt{\temp n}}\quad\mbox{ for any }t\geq T~.
\end{equation*}
We next wish to show that within $2\gamma n/\temp$ additional steps, $S_t$ and $\tS_t$ coalesce
with probability at least $1-c/\sqrt{\gamma}$ for some constant $c>0$.

Consider time $T$, and let $D_t \deq S_t - \tS_t$. Recall that we have already established that
\begin{equation}\label{eq-D-T-bound}
\E D_T \leq 10/\sqrt{\temp n}~,
\end{equation}
and assume without loss of generality that $D_T > 0$. We now run the magnetization chains $S_t$ and $\tS_t$ independently for $T \leq t \leq \tau_1$, where
$$\tau_1 \deq \min \left\{t\geq T: D_t \in \{0, \mbox{$-\frac{2}{n}$}\}\right\}~,$$
and let $\mathcal{F}_t$ be the $\sigma$-field generated by these two
chains up to time $t$. By Lemma \ref{lem-contracting}, we deduce that for sufficiently large values of $n$, if $D_t > 0$ then
\begin{align}\label{eq-D-t-decreasing}
\E[D_{t+1}-D_t \mid \F_t] \leq -\frac{\temp}{2n}D_t \leq 0~,
\end{align}
and $D_t$ is a supermartingale with respect to $\mathcal{F}_t$. Hence, so is
$$W_t \deq D_{T+t} \cdot \frac{n}{2} \one_{\{\tau_1 > t\}}~,$$
and it is easy to verify that $W_t$ satisfies the conditions of Lemma \ref{lem-supermatingale-positive}
(recall the upper bound on the holding probability of the magnetization chain, as well as the fact that at most one spin is updated at any given step). Therefore, for some constant $c > 0$,
\begin{align*}
\P\left(\tau_1 > t_n(2\gamma)\mid D_T\right)&=\P(W_0>0, W_1>0, \ldots, W_{t_n(2\gamma) - T} > 0 \mid D_T) \\
&\leq \frac{c n D_T}{\sqrt{\gamma n/\temp}}~.
\end{align*}
Taking expectation and plugging in \eqref{eq-D-T-bound}, we get that for some constant $c'$,
\begin{align}\label{eq-bound-tau1}
\P \left(\tau_1 > t_n(2\gamma)\right) \leq \frac{c'}{\sqrt{\gamma}}~.
\end{align}
From time $\tau_1$ and onward, we couple $S_t$ and $\tS_t$ using a monotone coupling, thus
$D_t$ becomes a non-negative supermartingale with $D_{\tau_1} \leq \frac{2}{n}$. By \eqref{eq-D-t-decreasing},
$$\E \left[D_{t+1}- D_t\mid \mathcal{F}_t\right] \leq -\frac{\temp}{n^{2}}~\mbox{ for $\tau_1\leq t < \taumag$}~,$$ and therefore, the Optional Stopping Theorem for non-negative supermartingales implies that, for some constant $c''$,
\begin{align}\label{eq-bound-taumag-tau1}
\P\left(\taumag - \tau_1 \geq n/\temp\right) \leq \frac{\E (\taumag - \tau_1)}{n/\temp} \leq \frac{c''}{\gamma}~.
\end{align}
Combining \eqref{eq-bound-tau1} and \eqref{eq-bound-taumag-tau1} we deduce that for some constant $c$,
$$\P\left(\taumag > t_n(3\gamma)\right) \leq \frac{c}{\sqrt{\gamma}}~,$$
completing the proof.
\end{proof}

\subsubsection{Lower bound}\label{sec:high-lower-mag}
We need to prove that the following statement holds for the distance of the magnetization at time $t$ from stationarity:
\begin{align}\label{eq-sub-lowerbound-mixing}
\lim_{\gamma \to\infty} \liminf_{n\to \infty} \;
d_n \left(\frac{1}{2}\cdot\frac{n}{\temp}\log (\temp^2 n) - \gamma
\frac{n}{\temp}\right) = 1~.
\end{align}
The idea is to show that, at time
$\frac{1}{2}\cdot\frac{n}{\temp}\log (\temp^2 n) - \gamma \frac{n}{\temp}$, the expected magnetization remains
large. Standard concentration inequalities will then imply that the magnetization will
typically be significantly far from $0$, unlike its stationary distribution.

To this end, we shall first analyze the third moment of the magnetization chain. Recalling the transition
rule \eqref{eq-magnet-transit} of $S_t$ under the notations \eqref{eq-def-p-+},\eqref{eq-def-p--}
\begin{align*}  p^{+}(s)
    & %= \frac{\mathrm{e}^{\beta s}}{\mathrm{e}^{\beta s} + \mathrm{e}^{-\beta s}}
    = \frac{1 + \tanh(\beta s )}{2}~,~
      p^{-}(s)
    %= \frac{\mathrm{e}^{-\beta s}}{\mathrm{e}^{\beta s} + \mathrm{e}^{-\beta s} }
    = \frac{1 - \tanh(\beta s)}{2}~,
\end{align*}
the following holds:
\begin{align}
\E &\left [ S_{t+1}^3 \mid S_t = s \right ] \nonumber\\
&= \frac{1+s}{2} p^{-}(s - n^{-1}) \left(s-\frac{2}{n}\right)^3+  \frac{1-s}{2} p^{+}(s + n^{-1}) \left(s+\frac{2}{n}\right)^3\nonumber\\
&\hspace{.5cm}+ \left(1 - \frac{1+s}{2} p^{-}(s - n^{-1}) - \frac{1-s}{2} p^{+}(s + n^{-1})\right) s^3 \nonumber\\
&= s^3 + \frac{6s^2}{n}\cdot \frac{1}{4}\Big(-2s + \tanh\left(\beta(s-n^{-1})\right)
+ \tanh\left(\beta(s+n^{-1})\right) \nonumber\\
&\hspace{.5cm} + s\left(\tanh\left(\beta(s-n^{-1})\right) - \tanh\left(\beta(s+n^{-1}\right)\right) \Big)
+c_1\frac{s}{n^2} + \frac{c_2}{n^3} ~.\label{eq-3rd-moment-1}
\end{align}
As $\tanh(x) \leq x$ for $x \geq 0$, for every $s > 0$ we get
\begin{align}
\E \left [ S_{t+1}^3 \mid S_t = s \right ] &\leq s^3 + \frac{3s^2}{2n}\left(-2s + \beta(s-n^{-1}) +
\beta(s+n^{-1})\right)
+c_1\frac{s}{n^2} + \frac{c_2}{n^3} \nonumber\\
&= s^3 - 3\frac{\temp}{n}s^3 + \frac{c_1}{n^2}s + \frac{c_2}{n^3}~.\label{eq-3rd-moment-2}
\end{align}
If $s = 0$, the above also holds, since in that case $|S_{t+1}|^3 \leq (2/n)^3$. Finally, by symmetry, if $s <
0$ then the distribution of $|S_{t+1}^3|=-S_{t+1}^3$ given $S_t=s$ is the same as that of
  $S_{t+1}^3$ given $S_t=|s|$, and altogether we get:
\begin{align*}
\E \left [ |S_{t+1}|^3 \mid S_t = s \right ] \leq  |s|^3 - 3\frac{\temp}{n}|s|^3 + \frac{c_1}{n^2}|s| +
\frac{c_2}{n^3}~.
\end{align*}
We deduce that
\begin{align}
  \E |S_{t+1}|^3 &\leq \E\left( |S_t|^3 - 3\frac{\temp}{n}|S_t|^3 +\frac{c_1}{n^2}|S_t| + \frac{c_2}{n^3} \right)
  \nonumber\\
  &\leq \left(1-\frac{3\temp}{n}\right)\E|S_t|^3 + \frac{c_1}{n^2}\E|S_t| + \frac{c_2}{n^3}~.\label{eq-3rd-moment-3}
\end{align}
Note that the following statement holds for the first moment of $S_t$:
\begin{align*}
&\E_{s_0} \left[|S_t|\right] \leq \sqrt{(\E_{s_0} S_t)^2 + \var_{s_0}(S_t)}\\
&\leq \sqrt{(s_t)^2 + \frac{16}{\temp n}} \leq  \left(1-\frac{\temp}{n}\right)^t|s_0| +
\frac{4}{\sqrt{\temp n}}~.\end{align*} Hence,
\begin{align*}
  \E_{s_0} |S_{t+1}|^3 &\leq \left(1-\frac{3\temp}{n}\right)\E_{s_0}|S_t|^3 +
  \frac{c_1}{n^2} \left(1-\frac{\temp}{n}\right)^t |s_0| + \frac{2}{n^2 \sqrt{\temp n}}
  + \frac{c_2}{n^3} \\
  &= \eta^3 \E_{s_0} |S_t|^3 + \eta^t\frac{c_1}{n^2}|s_0| + \frac{4}{n^2\sqrt{\temp n}} +\frac{c_2' \temp^2}{n^2}~,
\end{align*}
where $\eta = 1 - \temp/n$, and the extra error term involving $c_2'$ absorbs the change of coefficient of
$\E_{s_0} |S_t|^3$ and also the $1/n^3$ term. Iterating, we obtain
\begin{align}
  \E_{s_0} |S_{t+1}|^3 &\leq \eta^{3t} |s_0|^3 + \eta^t  \frac{c_1}{n^2}|s_0|\sum_{j=0}^t\eta ^{2j}
  + \left( \frac{c_1'}{n^2\sqrt{\temp n}}
  + \frac{c_2' \temp^2}{n^2}\right)\sum_{j=0}^t \eta^{3j}\nonumber \\
  &\leq \eta^{3t}|s_0|^3 + \eta^t \frac{c_1}{n^2}\cdot \frac{|s_0|}{1-\eta^2}
  +\left( \frac{c_1'}{n^2\sqrt{\temp n}}
  + \frac{c_2' \temp^2}{n^2}\right)\cdot\frac{1}{1-\eta^3} \nonumber \\
  &\leq \eta^{3t}|s_0|^3 + \eta^t \frac{c_1}{\temp n}|s_0|
  +\frac{c_1'}{(\temp n)^{3/2}}
  + \frac{c_2' \temp}{n}~.  \label{eq-3rd-moment-bound}
\end{align}
  Define $Z_t \deq |S_t|\eta^{-t}$, whence
  $Z_0 = |S_0| =|s_0|$. Recalling \eqref{eq-conditional-exp-S}, and combining the Taylor expansion of $\tanh(x)$ given in \eqref{eq-tanh-taylor} with the fact that $|\psi(s)| = O\left(s/n^2\right)$, we get that for $s>0$
  \begin{equation*}
   \E\left[|S_{t+1}|\given S_t =s\right]
      \geq \eta s - \frac{s^3}{2n} - \frac{s}{n^2}~.
  \end{equation*}
By symmetry, an analogous statement holds for $s<0$, and altogether we obtain that
  \begin{equation}\label{eq-lowerbound-moment-1}
    \E\left[|S_{t+1}|\given S_t \right]
      \geq \eta |S_t| - \frac{|S_t|^3}{2n} - \frac{|S_t|}{n^2}~.
  \end{equation}
\begin{remark*}Note that \eqref{eq-lowerbound-moment-1} in fact holds for any temperature, having followed from
the basic definition of the transition rule of $(S_t)$, rather than from any special properties that this chain
may have in the high temperature regime.\end{remark*}

Rearranging the terms and multiplying by $\eta^{-(t+1)}$, we obtain that
for any sufficiently large $n$,
\begin{align*}
  \E\Big[\Big(1-\frac{2}{n^{2}}\Big)Z_{t}-Z_{t+1} \mid S_t\Big] \leq \frac{1}{n}\eta^{-t}|S_t|^3~,
\end{align*}
where we used the fact that $\eta^{-1} \leq 2$. Taking expectation and plugging in \eqref{eq-3rd-moment-bound}, we deduce that
\begin{align}\label{eq-Zt-bound}
  \E_{s_0}\Big[\Big(1-\frac{2}{n^{2}}\Big)Z_t-Z_{t+1}\Big] \leq \frac{1}{n}\left(\eta^{2t}|s_0|^3
  + \frac{c_1}{\temp n}|s_0|
  +\eta^{-t}\left(\frac{c_1'}{(\temp n)^{3/2}}
  + \frac{c_2' \temp}{n}\right)\right).
\end{align}
Set $$\overline{t} = \frac{n}{2\temp} \log(\temp^2 n) - \gamma n/\temp~,$$ and notice that when $n$ is sufficiently
large, $\left(1-\frac{2}{n^{2}}\right)^{-(t+1)} \leq 2$ for any $t\leq \overline{t}$. Therefore,
multiplying \eqref{eq-Zt-bound} by $(1-\frac{2}{n^{2}})^{-(t+1)}$ and summing over gives:
\begin{align*}
|s_0| - 2\E_{s_0} Z_{\overline{t}}&\leq \frac{2|s_0|^3}{n(1-\eta^2)} + \overline{t} \frac{c_1}{\temp
n^2}|s_0|+\frac{2\eta^{-\overline{t}}}{n(1-\eta)}
 \left(\frac{c_1'}{(\temp n)^{3/2}} + \frac{c_2' \temp}{n}\right) \\
 &\leq \frac{2|s_0|^3}{\temp} + \frac{c_1 \log(\temp^2 n)}{2 \temp^2 n}|s_0| + \frac{c_1'}{\temp^{3/2} n } +
\frac{c_2}{\temp n^{3/2}}
  + \frac{c_2' \temp}{\sqrt{n}}\\
%\end{align*}
%\begin{align*}
%&  =   \frac{|s_0|^3}{\temp} + \frac{c_1 \log(\temp^2 n)}{2 \temp^2 n}|s_0|\\
%&\hspace{1cm}+ \sqrt{\temp}\left(\frac{c_1'}{\temp^2 n } + \frac{c_2}{(\temp n)^{3/2}}
 % + c_2'\sqrt{\frac{\temp}{n}}+O\left(\frac{\log(\temp^2 n)}{\temp^{3/2} n}\right)\right) \\
& =  \frac{2|s_0|^3}{\temp} + o(\sqrt{\temp}+ |s_0|)~,
\end{align*}
where the last inequality follows from the assumption $\temp^2 n \to \infty$. We now select $s_0 =
\sqrt{\temp}/3$, which gives
$$
  \sqrt{\temp}/3 - 2 \E_{s_0} Z_{\overline{t}} \leq 2\sqrt{\temp}/27 +o(\sqrt{\temp})~,
$$
and for a sufficiently large $n$ we get
\begin{equation*}
 \E_{s_0} Z_{\overline{t}} \geq  \sqrt{\temp}/9~.
\end{equation*}
Recalling the definition of $Z_t$, and using the well known fact that $(1-x) \geq \exp(-x/(1-x))$ for $0< x <
1$, we get that for a sufficiently large $n$,
\begin{equation}\label{eq-S-t-bar-bound}
 \E_{s_0} |S_{\overline{t}}| \geq \eta^{\overline{t}} \sqrt{\temp}/9
 \geq \frac{\mathrm{e}^{\gamma/2}}{10\sqrt{\temp n}} =: L ~.
\end{equation}
Lemma \ref{lem-var-bound} implies that $\max\{\var_{s_0}(S_t), \var_{\mu_n}(\tS_t)\} \leq 16/{\temp n}$. Therefore,
recalling that $\E_{\mu_n} \tS_{\overline{t}}=0$, Chebyshev's inequality gives
\begin{align*}
\P_{s_0}( |S_{\overline{t}}| \leq L/2 ) &\leq \P_{s_0}( \left||S_{\overline{t}}| - \E_{s_0} |S_{\overline{t}}|
\right| \geq L/2 )
\leq \frac{16/(\temp n)}{L^2/4}=c \mathrm{e}^{-\gamma}~, \\
\P_{\mu_n}( |\tS_{\overline{t}}| \geq L/2 ) &\leq \frac{16/(\temp n)}{L^2/4}=c \mathrm{e}^{-\gamma}~.
\end{align*}
Hence, letting $\pi$ denote the stationary distribution of $S_t$, and
 $A_L$ denote the set $\left[-\frac{L}{2}, \frac{L}{2}\right]$, we obtain that
  \begin{equation*}
    \| \P_{s_0}( S_{\overline{t}} \in \cdot ) - \pi \|_{{\rm TV}}
    \geq \pi(A_L)  - \P_{s_0}( |S_{\overline{t}}| \in A_L )
    \geq 1 - 2 c \mathrm{e}^{-\gamma}~,
  \end{equation*}
which immediately gives \eqref{eq-sub-lowerbound-mixing}. \qed

\subsection{Full Mixing of the Glauber dynamics}\label{sec:high-full-mixing}
In order to boost the mixing of the magnetization into the full mixing of the configurations, we will
need the following result, which was implicitly proved in \cite{LLP}*{Sections 3.3, 3.4} using a Two Coordinate
Chain analysis. Although the authors of \cite{LLP} were considering the case of $0 < \beta < 1$ fixed, one can
follow the same arguments and extend this result to any $\beta < 1$. Following is this generalization of their
result:
\begin{theorem}[\cite{LLP}]\label{thm-two-coord-chain-subc}
Let $(X_t)$ be an instance of the Glauber dynamics and $\mu_n$ the stationary distribution of the dynamics. Suppose $X_0$ is supported by
\begin{equation*}\label{eq-Omega-0-def}\Omega_0 \deq \{ \sigma \in \Omega : |S(\sigma)| \leq 1/2\}~.\end{equation*}
For any $\sigma_0 \in \Omega_0$ and $\tilde{\sigma} \in \Omega$, we consider the dynamics $(X_t)$ starting from
$\sigma_0$ and an additional Glauber dynamics $(\tilde{X}_t)$ starting from $\tilde{\sigma}$, and define:
\begin{align*}
&  \taumag \deq  \min\{t : S(X_t) = S(\tilde{X}_t)\}~,\\
& U(\sigma) \deq \left|\{i : \sigma(i) = \sigma_0(i) = 1\}\right|~,~ V(\sigma) \deq \left|\{i : \sigma(i) = \sigma_0(i) = -1\}\right|~,\\
 & \Xi \deq \left\{ \sigma: \min\{ U(\sigma), U(\sigma_0) - U(\sigma), V(\sigma), V(\sigma_0) - V(\sigma))  \} \geq n/20 \right\}~,\\
&  R(t)  \deq \left| U(X_t) - U(\tilde{X}_t) \right|~,\\
 & H_1(t) \deq \{ \taumag \leq t \}~,~ H_2(t_1,t_2) \deq \cap_{i=t_1}^{t_2} \{ X_i \in \Xi \wedge \tilde{X}_i \in \Xi \}~.
\end{align*}
Then for any possible coupling of $X_t$ and $\tX_t$, the following holds:
\begin{align}
\max_{\sigma_0 \in \Omega_0} &\| \P_{\sigma_0}(X_{r_2} \in \cdot) - \mu_n \|_\mathrm{TV} \leq\mathop{\max_{\sigma_0 \in
\Omega_0}}_{\tilde{\sigma}\in\Omega} \Big[\P_{\sigma_0,\tilde{\sigma}}
(\overline{H_1(r_1)}) \nonumber\\
& + \P_{\sigma_0,\tilde{\sigma}}(R_{r_1} > \alpha \sqrt{n/\temp}) + \P_{\sigma_0,\tilde{\sigma}}
(\overline{H_2(r_1,r_2)}) + \frac{\alpha c_1
}{\sqrt{r_2-r_1}}\cdot\sqrt{\frac{n}{\temp}}\Big]~,\label{eq-two-coordinate-chain-subc}
\end{align}
where $r_1 < r_2$ and $\alpha > 0$.
\end{theorem}

The rest of this subsection will be devoted to establishing a series of properties satisfied by the magnetization throughout the mildly subcritical case, in order to ultimately apply the above theorem.

First, we shall show that any instance of the Glauber dynamics concentrates on $\Omega_0$ once it
performs an initial burn-in period of $n/\temp$ steps. It suffices to show this for the dynamics started from $s_0=1$: to see this, consider a monotone-coupling
of the dynamics $(X_t)$ starting from an arbitrary configuration, together with two additional instances of the dynamics, $(X_t^+)$
starting from $s_0=1$ (from above) and $(X_t^-)$ starting from $s_0=-1$ (from below). By definition of the monotone-coupling, the chains $(X_t^+)$ and $(X_t^-)$ ``trap'' the chain $(X_t)$, and by symmetry it indeed remains to show that
$$\P_{1} (|S_{t_0}| \leq 1/2) = 1-o(1)~,\mbox{ where $t_0=n/\temp$}~.$$
Recalling \eqref{eq-st-one-step}, we have $s_{t+1} \leq (1 - \frac{\temp}{n}) s_t$ where
$s_t = \E\left[|S_{t}|\one_{\{\tau_0 > t\}}\right]$, thus
$$\E_1 \left[|S_{t_0}|\one_{\{\tau_0 > t_0\}}\right] \leq  \mathrm{e}^{-1}~.$$
Adding this to the fact that $\E_1 S_{t_0}\one\{\tau_0 \leq t_0\}=0$, which follows immediately from symmetry, we conclude that $\E_1 S_{t_0}\leq \mathrm{e}^{-1}$. Next, applying Lemma \ref{lem-var-bound} to our case and noting that
\eqref{eq.abscontr} holds for $\rho' = 1- \frac{1}{n}\big(1- n \tanh(\frac{\beta}{n})\big) \leq 1-\frac{\temp}{n}$, we conclude that
$$ \var(S_{t}) \leq \nu_1 \frac{n}{\temp} \leq \left(\frac{4}{n}\right)^2\frac{n}{\temp} = \frac{16}{\temp n}~\mbox{ for all $t$}~.$$
Hence, Chebyshev's inequality gives that $|S_{t_0}| \leq 1/2$ with high probability. We may therefore assume henceforth that our initial configuration already belongs to some good state $\sigma_0 \in \Omega_0$.

Next, set:
\begin{align*}
  T \deq t_n(\gamma)~,
  r_0 \deq t_n (2 \gamma)~,
  r_1 \deq t_n(3\gamma)~,
  r_2 \deq t_n(4\gamma)~.
\end{align*}
We will next bound the terms in the righthand side of \eqref{eq-two-coordinate-chain-subc} in order.
First, recall that Lemma \ref{lem-mag-coalesce-from-0} already provided us with a bound on the probability of $\overline{H_1(r_1)}$, by stating there for constant $c>0$
\begin{equation}
\label{eq-H1(r1)-bound}
\P(\taumag > r_1) \leq \frac{c}{\sqrt{\gamma}}~.
\end{equation}

Our next task is to provide an upper bound on $R_{r_1}$, and
namely, to show that it typically has order at most
$\sqrt{n/\temp}$. In order to obtain such a bound, we will analyze
the sum of the spins over the set $B \deq \{ i : \sigma_0(i) =
1\}$. Define $$M_t(B) \deq \frac{1}{2}\sum_{i\in B} X_t(i)~,$$ and
consider the monotone-coupling of $(X_t)$ with the chains
$(X_t^+)$ and $(X_t^-)$ starting from the all-plus and all-minus
positions respectively, such that $X_t^- \leq X_t \leq X_t^+$.
By defining $M_t^+$ and $ M_t^-$ accordingly, we get that
\begin{align*}
\E(M_t(B))^2 \leq \E (M_t^+(B))^2 + \E (M_t^-(B))^2 = 2\E
(M_t^+(B))^2~.
\end{align*}
By \eqref{eq-Est-bound-from-T}, we immediately get that for $t\geq
T$, $|\E M_t^+ (B)|\leq \sqrt{\frac{n}{\temp}}~.$ We will next
bound the variance of $M_t^+(B)$, by considering the following two cases:
\begin{enumeratei}
\item If every pair of spins of $X_t^+$ is positively correlated (since $X_0^+$ is the all-plus configuration, by symmetry, the covariances of each pair of spins is the same), then we can infer that $$\var (M_t^+(B)) \leq
\var\Big(\frac{1}{2}\sum_{i\in[n]}X_t^+(i)\Big) = \frac{n^2}{4}\var\left(S(X_t^+)\right)\leq \frac{4n}{\temp}~.$$
\item Otherwise, every pair of spins of $X_t^+$ is negatively correlated, and it
follows that $$\var (M_t^+(B)) \leq \sum_{i\in B
}\var\Big(\frac{1}{2}X_t^+(i)\Big) \leq \frac{n}{4}~.$$
\end{enumeratei}
Altogether, we conclude that for all $t \geq T$,
\begin{align}
  \E |M_{t}(B)| &\leq \sqrt{\E \left(M_{t}(B)\right)^2}
  %\leq \sqrt{2\E \left(M_{t}^+(B)\right)^2}
  \leq  \sqrt{2\var(M_t^+(B)+2\left(\E M_t(B)\right)^2}
  \nonumber\\  &\leq \sqrt{\frac{8n}{\temp} + \frac{2n}{\temp}} \leq 8 \sqrt{\frac{n}{\temp}}~.\label{eq-M-T-bound}
\end{align}
This immediately implies that
\begin{equation*}
\E R_{r_1}= \E|M_{r_1}(B) - \tilde{M}_{r_1}(B)|\leq\E|M_{r_1}(B)| +\E|\tilde{M}_{r_1}(B)|  \leq
16\sqrt{\frac{n}{\temp}}~,
\end{equation*}
and an application of Markov's inequality now gives
\begin{equation}\label{eq-R_r-bound}
\P (R_{r_1} \geq \alpha \sqrt{\frac{n}{\temp}}) \leq \frac{16}{\alpha}~.
\end{equation}

It remains to bound the probability of $\overline{H_2(r_1,r_2)}$. Define:
  \begin{align*}
    Y& \deq \sum_{r_1 \leq t \leq r_2} \one\{ |M_t(B)| > n/64 \} ~,
  \end{align*}
and notice that
\begin{equation*}
\P \bigg(\bigcup_{t=r_1}^{r_2} \left\{ |M_t(B)| \ge n/32 \right\}\bigg)
      \leq \P(Y > n/64) \leq \frac{c_0 \E [Y]}{n} ~.
\end{equation*}
Recall that the second inequality of \eqref{eq-M-T-bound} actually gives $\E |M_t(B)|^2 \leq \frac{5n}{\temp}$. Hence, a standard second moment argument gives
\begin{equation*}
  \P(|M_t(B)| > n/64) = O\left(\frac{1}{\temp n}\right)~.
\end{equation*}
Altogether, $\E_{\sigma_0}Y = O(1/\temp^{2})$ and
  \begin{equation*}
    \P_{\sigma_0} \bigg(\bigcup_{t=r_1}^{r_2} \left\{ |M_t(B)| \ge n/32 \right\}\bigg) = O\left(\frac{1}{\temp^2 n}\right)~.
  \end{equation*}
Applying an analogous argument to the chain $(\tilde{X}_t)$, we obtain that
  \begin{equation*}
    \P_{\tilde{\sigma}}\bigg(\bigcup_{t=r_1}^{r_2} \left\{ |\tilde{M}_t(B)| \ge n/32 \right\}\bigg) = O\left(\frac{1}{\temp^2 n}\right)~,
  \end{equation*}
and combining the last two inequalities, we conclude that
\begin{equation}\label{eq-H-2-bound}
\P_{\sigma_0,\tilde{\sigma}} \left(\overline{H_2(r_1,r_2)}\right) = O\left(\frac{1}{\temp^2 n}\right)~.
\end{equation}

Finally, we have established all the properties needed in order to apply Theorem \ref{thm-two-coord-chain-subc}. At the cost of a negligible number of burn-in steps, the state of $(X_t)$ with high probability belongs to $\Omega_0$. We may thus plug in \eqref{eq-H1(r1)-bound}, \eqref{eq-R_r-bound} and \eqref{eq-H-2-bound} into Theorem \ref{thm-two-coord-chain-subc}, choosing $\alpha = \sqrt{\gamma}$, to obtain \eqref{eq-sub-upperbound-mixing}.

\subsection{Spectral gap Analysis}\label{sec:high-spectral}
By Proposition \ref{prop-spetral-glauber-mag}, it suffices to determine the spectral gap of the magnetization
chain. The lower bound will follow from the next lemma of \cite{Chen} (see also \cite{LPW}*{Theorem 13.1}) along with the contraction properties of the magnetization chain.
\begin{lemma}[\cite{Chen}]\label{lem-spectra-contraction}
Suppose $\Omega$ is a metric space with distance $\rho$. Let P be a transition matrix for a Markov chain, not
necessarily reversible. Suppose there exists a constant $\theta < 1$ and for each $x, y \in \Omega$, there is a
coupling $(X_1, Y_1)$ of $P(x, \cdot)$ and $P(y, \cdot)$ satisfying
\begin{equation*}
\E_{x,y} (\rho(X_1, Y_1))\leq \theta \rho(x, y)~.
\end{equation*}
If $\lambda$ is an eigenvalue of $P$ different from $1$, then $|\lambda| \leq \theta$. In particular, the spectral gap satisfies $\gap \geq 1- \theta$.
\end{lemma}
Recalling \eqref{eq-high-temp-contraction}, the monotone coupling of $S_t$ and $\tS_t$ implies that
\begin{equation*}
\E_{s, \tilde{s}} \big|S_{1} - \tilde{S}_1\big| \leq \Big(1- \frac{\temp}{n} + o\Big(\frac{\temp}{n}\Big)\Big)\left| s-\tilde{s}\right|~.
\end{equation*}
Therefore, Lemma \ref{lem-spectra-contraction} ensures that $\gap \geq (1+ o(1)) \frac{\temp}{n}$.

It remains to show a matching upper bound on $\gap$, the spectral gap of the magnetization chain. Let $M$ be the transition kernel of this chain, and $\pi$ be its stationary distribution. Similar to our final argument in Proposition \ref{prop-spetral-glauber-mag} (recall \eqref{eq-dirichlet-form}), we apply the Dirichlet representation for the
spectral gap (as given in \cite{LPW}*{Lemma 13.7}) with respect to the function $f$ being the identity map
$\mathds{1}$ on the space of normalized magnetization, we obtain that
\begin{align}\label{eq-gap-dirichlet-high}
\gap \leq  \frac{\mathcal{E}(\mathds{1})} {\left<\mathds{1},\mathds{1}\right>_{\pi}}= \frac{
\left<(I-M)\mathds{1},\mathds{1}\right>_\pi}{\left<\mathds{1},\mathds{1}\right>_{\pi}}
%= 1 - \frac{
%\left<M\mathds{1},\mathds{1}\right>_\pi}{\left<\mathds{1},\mathds{1}\right>_{\pi}}
= 1 - \frac{\E_\pi \left[\E\left[S_t S_{t+1} \mid S_t\right]\right] }
{\E_\pi S_t^2 }~,
\end{align}
where $\E_\pi S_t^k$ is the $k$-th moment of the stationary magnetization chain $(S_t)$.
Recall \eqref{eq-lowerbound-moment-1} (where $\eta = 1-\frac{\temp}{n}$), and notice that the following slightly stronger inequality in fact holds:
\begin{align*}
\E\left[\operatorname{sign}(S_t) S_{t+1} \given S_t \right]
\geq \eta |S_t| - \frac{|S_t|^3}{2n} - \frac{|S_t|}{n^2}~.
\end{align*}
(to see this, one needs to apply the same argument that led to \eqref{eq-lowerbound-moment-1}, then verify the special cases $S_t\in\{0,\frac{1}{n}\}$). It thus follows that
\begin{align*}
\E\left[S_t S_{t+1} \given S_t \right]
\geq \eta S_t^2 - \frac{S_t^4}{2n} - \frac{S_t^2}{n^2}~,
\end{align*}
and plugging the above into \eqref{eq-gap-dirichlet-high} we get
\begin{align}\label{eq-sub-gap-upperbound}
\gap \leq \frac{\temp}{n} + \frac{1}{2n} \cdot\frac{\E_{\pi} S_t^4}{\E_{\pi} S_t^2} + \frac{1}{n^2}~.
\end{align}
In order to bound the second term in \eqref{eq-sub-gap-upperbound}, we need to give an upper bound for the fourth moment in terms of the second moment. The next argument is similar to the one used earlier to bound the third moment of the magnetization chain (see \eqref{eq-3rd-moment-1}), and hence will be described in a more concise manner.

{
\newcommand{\tanhplus}{h^+}
\newcommand{\tanhminus}{h^-}
For convenience, we use the abbreviations $\tanhplus \deq \tanh\left(\beta(s+n^{-1})\right)$ and $\tanhminus \deq \tanh\left(\beta(s-n^{-1})\right)$. By definition (see \eqref{eq-magnet-transit}) the following then holds:
\begin{align*}
\E [S_{t+1}^4 &\mid S_{t} = s]
=s^{4} + \frac{2}{n}s^3 \left(-2s + \tanhminus + \tanhplus + s\tanhminus- \tanhplus\right)\\
&+\frac{6}{n^2}s^{2} \left(2+ \tanhplus - \tanhminus - s\tanhminus + \tanhplus \right)\\
&+ \frac{8}{n^3}s^3 \left(-2s + \tanhminus + \tanhplus + s\tanhminus-\tanhplus\right)\\
&+ \frac{4}{n^4} \left(2+ \tanhplus - \tanhminus - s\tanhminus + \tanhplus \right)\\
&\qquad\qquad\leq \left(1- \frac{4\temp}{n}\right)s^4 + \frac{12}{n^2}s^2 + \frac{16}{n^4}~.
\end{align*}
Now, taking expectation and letting the $S_t$ be distributed according to $\pi$, we obtain that
$$ \E_\pi S_t^4 \leq \frac{3}{\temp n}\E_\pi S_t^2 + \frac{4}{\temp n^3}~.$$
Recalling that, as the spins are positively correlated, $\var_\pi (S_t) \geq \frac{1}{n}$, we get
\begin{equation}
\E_{\pi} S_t^4 \leq \left(3+\frac{4}{n}\right)\frac{\E_\pi S_t^2}{n\temp} ~.\label{eq-4th-moment}
\end{equation}
Plugging \eqref{eq-4th-moment} into \eqref{eq-sub-gap-upperbound}, we conclude that
$$\gap \leq \frac{\temp}{n}\left(1+ O\Big(\frac{1}{\temp^2 n}\Big)\right) = (1+ o(1))\frac{\temp}{n}~.$$
}

\section{The critical window}\label{sec:critical}
In this section we prove Theorem \ref{thm-critical-temp}, which establishes that the
critical window has a mixing-time of order $n^{3/2}$ without a cutoff, as well as a spectral-gap of order $n^{-3/2}$.

\subsection{Upper bound}\label{sec:critical-upper}
%Our argument in this case follows the arguments of \cite{LLP}, combined with Theorem \ref{thm-tau-0}.
Let $(X_t)$ denote the Glauber dynamics, started from an arbitrary
configuration $\sigma$, and let $(\tX_t)$ denote the dynamics started from the stationary distribution $\mu_n$. As usual, let $(S_t)$ and $(\tS_t)$ denote the (normalized) magnetization chains of $(X_t)$ and $(\tX_t)$ respectively.

Let $\epsilon > 0$. The case $\delta^2 n=O(1)$ of Theorem \ref{thm-tau-0} implies that, for a sufficiently large $\gamma > 0$, $\P_\sigma\left(\tau_0 \geq \gamma n^{3/2}\right) < \epsilon$.
Plugging this into Lemma \ref{lem-taumag-tau-0}, we deduce that there exists some $c_\epsilon > 0$, such that the chains $S_t$ and $\tS_t$ coalesce after at most $c_\epsilon n^{3/2}$ steps with probability at least $1-\epsilon$.

At this point, Lemma \ref{lem-fullmix-taumag} implies that $(X_t)$ and $(\tX_t)$ coalesce after at most $O(n^{3/2}) + O(n\log n) = O(n^{3/2})$ additional steps with probability arbitrarily close to $1$, as required.

\subsection{Lower bound}\label{sec:critical-lower}
Throughout this argument, recall that $\temp$ is possibly negative, yet satisfies $\temp^2 n = O(1)$. By
\eqref{eq-3rd-moment-3},
\begin{align*}
  \E |S_{t+1}|^3 &\leq \E\left( |S_t|^3 - 3\frac{\temp}{n}|S_t|^3 +\frac{c_1}{n^2}|S_t| + \frac{c_2}{n^3} \right)\\
  &\leq \left(1-\frac{3\temp}{n}\right)\E|S_t|^3 + \frac{c_1}{n^2}\E|S_t| + \frac{c_2}{n^3}~.
\end{align*}
Recalling Lemma \ref{lem-contracting}, and plugging the fact that $\temp = O(n^{-1/2})$ in \eqref{eq-contraction}, the following holds. If $S_t$ and $\tS_t$ are the magnetization chains corresponding to two instances of the Glauber dynamics, then for some constant $c>0$ and any sufficiently large $n$,
\begin{align}\label{eq-critical-temp-contraction}
\E_{s,\tilde{s}}|S_1 - \tS_1| \leq (1+ c n^{-3/2})|s-\tilde{s}|~.
\end{align}
Combining this with the extended form of Lemma \ref{lem-var-bound}, as given in \eqref{eq-var-bound-rho-geq-1}, we deduce that if $t \leq \epsilon n^{3/2}$ for some small fixed $\epsilon>0$, then
$\var_{s_0}S_t \leq 4t/n^2$. Therefore,
\begin{equation*} \E_{s_0} \left[|S_t|\right] \leq \sqrt{|\E_{s_0} S_t|^2 + \var_{s_0}S_t}\leq \left(1-\frac{\temp}{n}\right)^t|s_0| + \frac{2\sqrt{t}}{n}~.\end{equation*}
Therefore,
\begin{align*}
  \E_{s_0} |S_{t+1}|^3 &\leq \left(1-\frac{3\temp}{n}\right)\E_{s_0}|S_t|^3 +
  \frac{c_1}{n^2} \left(1-\frac{\temp}{n}\right)^t |s_0| + \frac{c_1'\sqrt{t}}{n^3} \\
  &\leq \eta^3 \E_{s_0} |S_t|^3 + \eta^t\frac{c_1}{n^2}|s_0| + \frac{c_1'\sqrt{t}}{n^3} ~,
\end{align*}
where again $\eta = 1 - \temp/n$. Iterating, we obtain
\begin{align}
  \E_{s_0} |S_{t+1}|^3 &\leq \eta^{3t} |s_0|^3 + \eta^t  \frac{c_1}{n^2}|s_0|\sum_{j=0}^t\eta ^{2j}
  + \frac{c_1'\sqrt{t}}{n^3} \sum_{j=0}^t \eta^{3j}\nonumber \\
  &\leq \eta^{3t}|s_0|^3 + \eta^t \frac{c_1}{n^2}\cdot \frac{\eta^{2t-1}-1}{\eta^2-1} |s_0|
  +\frac{c_1'\sqrt{t}}{n^3}\cdot \frac{\eta^{3t}-1}{\eta^3-1}\nonumber \\
  &\leq \eta^{3t}|s_0|^3 + \eta^t \frac{c_1}{n^2}\cdot 2t |s_0|
  +\frac{c_1'\sqrt{t}}{n^3}\cdot 3t~,
\end{align}
where the last inequality holds for sufficiently large $n$ and $t\leq \epsilon n^{3/2}$ with $\epsilon>0$ small enough
(such a choice ensures that $\eta^t$ will be suitably small). Define $Z_t \deq |S_t|\eta^{-t}$, whence
  $Z_0 = |S_0| =|s_0|$.  Applying \eqref{eq-lowerbound-moment-1} (recall that it holds for any temperature) and using the fact that $\eta^{-1} \leq 2$, we get
  \begin{equation*}
    \E[Z_{t+1} \mid S_t] \ge Z_t- \frac{1}{n}\left(\eta^{-t}|S_t|^3+O(1/n)\right),
  \end{equation*}
  for $n$ large enough, hence
\begin{align*}
  \E[Z_{t}-Z_{t+1} \mid S_t] \leq \frac{1}{n}\left(\eta^{-t}|S_t|^3 + O(1/n)\right)~.
\end{align*}
Taking expectation and plugging in \eqref{eq-3rd-moment-bound},
\begin{align}\label{eq-Zt-bound-critical}
  \E_{s_0}[Z_t-Z_{t+1}] \leq \frac{1}{n}\left(\eta^{2t}|s_0|^3 + \frac{c_2 t}{n^2} |s_0|
  +\eta^{-t}\frac{c_2' t^{3/2}}{n^3} + O(1/n)\right)~.
\end{align}
Set $\overline{t} = n^{3/2}/A^4$ for some large constant $A$ such that $\frac{1}{2}\leq \eta^{\overline{t}} \leq 2$. Summing
over \eqref{eq-Zt-bound-critical} we obtain that
\begin{align*}
  |s_0| - \E_{s_0} Z_{\overline{t}} &\leq \frac{1-\eta^{2\overline{t}}}{n(1-\eta^2)}|s_0|^3 + \overline{t}^2 \frac{c_2}{n^3}|s_0|
 +2\eta^{-\overline{t}}\cdot \overline{t}^{5/2}/n^4 +O(\overline{t}/n^2)\\
&  \leq \frac{2}{A^{4}} \sqrt{n}|s_0|^3 + \frac{c_2}{A^{8}} |s_0| + \frac{2}{A^{10}}\mathrm{e}^{ \sqrt{\temp^2 n}/A^4} n^{-1/4} +
O(n^{-1/2})~.
\end{align*}
We now select $s_0 = A n^{-1/4}$ for some large constant $A$; this gives
$$
  A n^{-1/4} - \E_{s_0} Z_{\overline{t}} \leq \left(\frac{2}{A} + \frac{c_2}{A^7}
 +  \frac{2}{A^{10}} \mathrm{e}^{\sqrt{\temp^2 n}/A^4}
 \right) n^{-1/4} + O(n^{-1/2})~.
$$
Choosing $A$ large enough to swallow the constant $c_2$ as well as the term $\temp^2 n$ (using the fact that $\temp^2 n$
is bounded), we obtain that
$$
 \E_{s_0} Z_{\overline{t}} \geq  \frac{1}{2}A n^{-1/4}~.
$$
Translating $Z_t$ back to $|S_t|$, we obtain
\begin{equation}\label{eq-S-t-bar-bound-2}
 \E_{s_0} |S_{\overline{t}}| \geq \eta^{\overline{t}} \cdot \frac{1}{2}An^{-1/4}
 \geq \sqrt{A} n ^{-1/4} =: B ~,
\end{equation}
provided that $A$ is sufficiently large (once again, using the fact that $\eta^{\overline{t}}$ is bounded, this
time from below). Since
\begin{equation}\label{eq-var-tbar-bound}\var_{s_0}(S_{\overline{t}}) \leq 16 \overline{t}/n^2 =
\frac{16}{A^4} n^{-1/2}~,\end{equation} the following concentration result on the stationary chain $(\tS_t)$ will
complete the proof:
\begin{equation}\label{eq-con-mu-critical}
\P_{\mu_n}(|\tS_{t}| \geq A n^{-1/4})\} \leq ~ \epsilon(A)~,\mbox{ and }\lim_{A\to\infty}\epsilon(A)=0~.\end{equation}
Indeed, combining the above two statements, Chebyshev's inequality implies that
  \begin{align}\label{eq-critical-tv-lowerbound}
    \| \P_{s_0}( S_{\overline{t}} \in \cdot ) - \pi \|_{{\rm TV}}
    &\geq \pi([-B/2,B/2])  - \P_{s_0}( |S_{\overline{t}}| \leq B/2 )\nonumber
    \\
    &\geq 1 - \frac{64}{A^5} - \epsilon(\sqrt{A})~.
  \end{align}
It remains to prove \eqref{eq-con-mu-critical}. Since we are proving a lower bound for the mixing-time, it
suffices to consider a sub-sequence of the $\temp_n$-s such that $\temp_n \sqrt{n}$ converges to some constant
(possibly 0). The following result establishes the limiting stationary distribution of the magnetization chain in this case.
\begin{theorem}
Suppose that $\lim_{n\to\infty}\temp_n \sqrt{n} = \alpha \in \mathbb{R}$. The following holds:
\begin{equation} \label{eq-norm-mag-convergence} \frac{S_{\mu_n}}{n^{-1/4}} \to
   \exp\left(-\frac{s^4}{12}- \alpha \frac{s^2}{2} \right)~.\end{equation}
\end{theorem}
\begin{proof}
We need the following theorem:
\begin{theorem}[\cite{EN}*{Theorem 3.9}]\label{thm-critical-magnetization}
Let $\rho$ denote some probability measure, and let $S_n(\rho) = \frac{1}{n}\sum_{j = 1}^n X_j(\rho)$, where the
$\{X_j(\rho):j\in[n]\}$ have joint distribution
$$\frac{1}{Z_n}\exp\left[\frac{\left(x_1+\ldots+x_n\right)^2}{2n}\right]\prod_{j=1}^n d\rho(x_j)~,$$
and $Z_n$ is a normalization constant. Suppose that $\{\rho_n : n=1,2,\ldots\}$ are measures satisfying
\begin{equation}\label{eq-rho-n-convergence-cond}
\exp(x^2/2)d \rho_n \to \exp(x^2/2)d\rho  ~.
\end{equation}
Suppose further that $\rho$ has the following properties:
\begin{enumerate}
  \item \emph{Pure}: the function $$G_\rho(s) \deq \frac{s^2}{2} - \log \int \mathrm{e}^{sx}d \rho(x)$$ has a unique global minimum.
  \item \emph{Centered at $m$}: let $m$ denote the location of the above global minimum.
  \item \emph{Strength $\temp$ and type $k$}: the parameters $k,\temp > 0$ are such that
  $$G_\rho(s) = G_\rho(m) + \temp \frac{(s-m)^{2k}}{(2k)!} + o((s-m)^{2k})~,$$ where the $o(\cdot)$-term
  tends to $0$ as $s\to m$.
\end{enumerate}
If, for some real numbers $\alpha_1,\ldots,\alpha_{2k-1}$ we have
  $$ G_{\rho_n}^{(j)}(m) = \frac{\alpha_j}{n^{1-j/2k}} + o(n^{-1+j/2k})~,\quad j=1,2,\ldots,2k-1~,n\to\infty~,$$
then the following holds:
\begin{equation*}
  S_n(\rho_n)\to\one_{\{s \neq m\}}
\end{equation*}
and
\begin{equation*}
  \frac{S_n(\rho_n)-m}{n^{-1/2k}} \to \left\{\begin{array}{cl}
    \displaystyle{N\left(-\frac{\alpha_1}{\temp},\frac{1}{\temp}-1\right)}, & \mbox{if }k=1~,\\
\noalign{\medskip}
   \displaystyle{\exp\left(-\temp \frac{s^2k}{(2k)!}-\sum_{j=1}^{2k-1}\alpha_j \frac{s^j}{j!} \right)}, & \mbox{if }k\geq 2~.\\
  \end{array}\right.~,
\end{equation*}
where $\temp^{-1}-1 > 0$ for $k=1$.
\end{theorem}
Let $\rho$ denote the two-point uniform measure on $\{-1,1\}$, and let $\rho_n$ denote the two-point uniform
measure on $\{-\beta_n,\beta_n\}$. As $|1-\beta_n| = \temp_n = O(1/\sqrt{n})$, the convergence requirement
\eqref{eq-rho-n-convergence-cond} of the measures $\rho_n$ is clearly satisfied. We proceed to verify the
properties of $\rho$:
\begin{align*}G_\rho(s) = \frac{s^2}{2}-\log \int \mathrm{e}^{sx} d \rho(x) =
\frac{s^2}{2} - \log \cosh(s)= \frac{s^4}{12}-\frac{s^6}{45} + O(s^8)~.
\end{align*}
This implies that $G_\rho$ has a unique global minimum at $m=0$, type $k=2$ and strength $\temp=2$. As $\temp_n
\sqrt{n} \to \alpha$, we deduce that the $G_{\rho_n}$-s satisfy
\begin{align*}G_{\rho_n} (s) &=\frac{s^2}{2} - \log \cosh(\beta_n s)~,\\
G^{(1)}_{\rho_n} (0) &= G^{(3)}_{\rho_n} (0) = 0~,\\
G^{(2)}_{\rho_n}(0) &= 1-\beta_n^2 = \temp_n(2-\temp_n) = \frac{2\alpha}{\sqrt{n}} + o(n^{-1/2})~.
\end{align*}
This completes the verification of the conditions of the theorem, and we obtain that
\begin{equation}\label{eq-S-convergence}
  \frac{S_n(\rho_n)}{n^{-1/4}} \to
   \exp\left(-\frac{s^4}{12}- \alpha \frac{s^2}{2} \right)~.
\end{equation}
Recalling that, if $x_i =\pm 1$ is the $i$-th spin,
\begin{equation} \label{eq.gibbs}
  \mu_n(x_1,\ldots,x_n)
    = \frac{1}{Z(\beta)}
      \exp\Big( \frac{\beta}{n}\sum_{1 \leq i < j \leq n}
        x_i x_j \Big),
\end{equation}
clearly $S_n(\rho_n)$ has the same distribution as $S_{\mu_n}$ for any $n$. This completes the proof of the
theorem.
\end{proof}

\begin{remark*}
One can verify that the above analysis of the mixing time in the critical window holds also for the censored dynamics (where the magnetization is restricted to be non-negative, by flipping all spins whenever it becomes negative). Indeed, the upper bound immediately holds as the censored dynamics is a function of the original Glauber dynamics. For the lower bound, notice that our argument tracked the absolute value of the magnetization chain, and hence can readily be applied to the censored case as-well. Altogether, the censored dynamics has a mixing time of order $n^{3/2}$ in the critical window $1\pm\temp$ for $\temp=O(1/\sqrt{n})$.
\end{remark*}

\subsection{Spectral gap analysis}\label{sec:critical-spectral}

The spectral gap bound in the critical temperature regime is obtained by combining the above analysis with results of \cite{DLP} on birth-and-death chains.

The lower bound on $\gap$ is a direct consequence of the fact that the mixing time has order $n^{3/2}$, and that the inequality $\trel \leq \tmix\big(\frac{1}{4}\big)$ always holds. It remains to prove the matching bound $\tmix\left(\frac{1}{4}\right) = O(\trel)$. Suppose that this is false, that is,
 $\trel = o\left(\tmix\big(\frac{1}{4}\big)\right)$.

 %We then study the magnetization chain, which is another birth-and-death chain, in the space of
%the normalized sum of the spins.
%We will use some of the results in Section 3 and for the sake of convenience, we recall several notations as
%follow.
Let $A$ be some large constant, and let $s_{0} = A n^{-1/4}$. %and $B= \sqrt{A} n^{-1/4}$.
Notice that the case $\temp^2 n = O(1)$ in Theorem \ref{thm-tau-0} implies that $\E_{1} \tau_{0} = O (n^{3/2})$. Furthermore, by Theorem \ref{thm-critical-magnetization}, there exists a strictly positive function of $A$,
$\epsilon(A)$, such that $\lim_{A\to\infty}\epsilon(A)=0$ and
$$\frac{1}{2}\epsilon(A)\leq \pi(S\geq s_{0}) \leq 2\epsilon(A) $$
for sufficiently large $n$. Applying Lemma \ref{lem-hitting-ratio} with $\alpha= \pi(S\geq s_0)$ and $\beta =
\frac{1}{2}$ gives $\E_{s_0} \tau_0 = o (n^{3/2})$. As in Subsection \ref{sec:critical-lower}, set $\bar{t} =
n^{3/2}/A^4$ for some large constant $A$. Combining Lemma \ref{lem-tv-hitting-bound} with Markov's inequality
gives the following total variation bound for this birth-and-death chain:
\begin{equation}\label{eq-tv-upperbound}
\|\P_{s_0}(S_{\bar{t}}\in \cdot) - \pi\|_{\mathrm{TV}} \leq 4 \epsilon(A) + o(1)~.
\end{equation}
However, the lower bound \eqref{eq-critical-tv-lowerbound} obtained in Subsection \ref{sec:critical-lower} implies that:
\begin{equation}\label{eq-tv-lowerbound}
\|\P_{s_0}(S_{\bar{t}}\in \cdot) - \pi\|_{\mathrm{TV}} \geq 1 - 4 \epsilon(\sqrt{A}/2) - 64 /A^5~.
\end{equation}
Choosing a sufficiently large constant $A$, \eqref{eq-tv-upperbound} and \eqref{eq-tv-lowerbound} together lead
to a contradiction for large $n$. We conclude that $\gap = O(n^{-3/2})$, completing the proof.

Note that, as the condition $\gap \cdot \tmix(\frac{1}{4})\to\infty$ is necessary for cutoff in any family of ergodic reversible finite Markov chains (see, e.g., \cite{DLP}), we immediately deduce that there is no cutoff in this regime.

\begin{remark*}
It is worth noting that the order of the spectral gap at $\beta_c=1$ follows from a simpler argument. Indeed, in that case, the \emph{upper bound} on $\gap$ can alternatively be derived from its Dirichlet representation, similar to the argument that appeared in the proof of Proposition \ref{prop-spetral-glauber-mag} (where we substitute the identity function, i.e., the sum of spins in the Dirichlet form).
For this argument, one needs a lower bound for the variance of the stationary magnetization. Such a bound is known for $\beta_c=1$ (see \cite{ENR}), rather than throughout the critical window.
\end{remark*}

\section{Low temperature regime}\label{sec:low}
In this section we prove Theorem \ref{thm-low-temp}, which establishes the order of the mixing time
and the spectral gap in the super critical regime (where the mixing of the dynamics is exponentially slow and there is no cutoff).

\subsection{Exponential mixing}
Recall that the normalized magnetization chain $S_t$ is a birth-and-death chain on the space $\magspace = \{-1, -1+\frac{2}{n},\ldots, 1-\frac{2}{n}, 1\}$, and for simplicity, assume throughout the proof that $n$ is even
(this is convenient since in this case we can refer to the $0$ state. Whenever $n$ is odd, the same proof holds
by letting $\frac{1}{n}$ take the role of the $0$ state).

The following notation will be useful. We define $$\magspace[a, b]\deq
\{x\in \magspace: a\leq x \leq b\}~,$$ and similarly define $\magspace (a, b)$, etc. accordingly.
For all $x\in\magspace$, let $p_x,q_x,h_x$ denote the transition probabilities of $S_t$ to the right, to the left and to
itself from the state $x$, that is:
\begin{align*}& p_x \deq P_M\left(x, x+\mbox{$\frac{2}{n}$}\right)= \frac{1-x}{2}\cdot\frac{1 + \tanh(\beta(x+n^{-1}))}{2}~,\\
&q_x \deq P_M\left(x, x-\mbox{$\frac{2}{n}$}\right) = \frac{1+x}{2}\cdot\frac{1 - \tanh(\beta(x-n^{-1}))}{2}~,\\
&h_x \deq P_M\left(x, x\right) = 1- p_x - q_x~.
\end{align*}
By well known results on birth-and-death chains (see, e.g., \cite{LPW}), the resistance $r_x$ and conductance
$c_x$ of the edge $(x, x+2/n)$, and the conductance $c'_x$ of the self-loop of vertex $x$ for $x\in
\magspace[0,1]$ are (the negative parts can be obtained immediately by symmetry)
\begin{equation}\label{eq-rk-resistence}
r_x = \prod_{y\in \magspace(0, x]} \frac{q_y}{p_{y}}~,~c_x = \prod_{y\in \magspace(0, x]}
\frac{p_y}{q_y}~,~ c'_x = \frac{h_x}{p_x + q_x} (c_{x-2/n} + c_x)~,
\end{equation}
and the commute-time between $x$ and $y$, $C_{x,y}$ for $x<y$ (the minimal time it takes the chain, starting
from $x$, to hit $y$ then return to $x$) satisfies
\begin{equation}
  \label{eq-commute-time-formula}
  \E C_{x,y} = 2 c_S R(x \leftrightarrow y)~,
\end{equation}
where
$$c_S \deq \sum_{x\in\magspace} (c_x +c'_x)\quad\hbox{and }\quad R(x \leftrightarrow y) \deq \sum_{z\in \magspace[x,
y)}r_z~.$$
Our first goal is to estimate the expected commute time between $0$ and $\zeta$. This is incorporated
in the next lemma.
\begin{lemma}
  \label{lem-commute-time}
  The expected commute time between $0$ and $\zeta$ has order \begin{equation}
    \label{eq-commute-time-estimate}
    \texp \deq \frac{n}{\temp}\exp\left(\frac{n}{2}\int_0^\zeta \log \frac{1+g(x)}{1-g(x)}\right)dx~,
  \end{equation}
  where $g(x)\deq\left(\tanh(\beta x)-x\right)/\left(1-x\tanh(\beta x)\right)$. In particular, in the special case $\temp\to 0$ we have $\E C_{0,\zeta} = \frac{n}{\temp}\exp \left((\frac{3}{4}+o(1))\temp^2 n\right)$, where the $o(1)$-term tends to $0$ as $n\to\infty$.
\end{lemma}
\begin{remark*}
If $\zeta \not\in \magspace$, instead we simply choose a state in $\magspace$ which is the nearest possible to $\zeta$. For a sufficiently large $n$, such a negligible adjustment would keep our calculations and arguments in tact. For the convenience of notation, let $\zeta$ denote the mentioned state in this case as well.
\end{remark*}
To prove Lemma \ref{lem-commute-time}, we need the following two lemmas, which establish
the order of the total conductance and effective resistance respectively.
\begin{lemma}\label{lem-order-conductance}
The total conductance satisfies
$$c_S = \Theta\left(\sqrt{\frac{n}{\temp}}\exp\left(\frac{n}{2}\int_{0}^{\zeta} \log \left(\frac{1+g(x)}{1-g(x)}\right)dx \right)\right)~.$$
\end{lemma}
\begin{lemma}\label{lem-order-resistance}
The effective resistance between $0$ and $\zeta$ satisfies
$$R(0\leftrightarrow \zeta) = \Theta(\sqrt{n/\temp})~.$$
\end{lemma}
\begin{proof}[\textbf{\emph{Proof of Lemma \ref{lem-order-conductance}}}]
Notice that for any $x \in \magspace$, the holding probability $h_x$ is uniformly bounded from below, and thus
$c'_x$ can be uniformly bounded from above by $(c_x+c_{x-2/n})$. It therefore follows that $c_S = \Theta (\tilde{c}_S)$
where $\tilde{c}_S \deq \sum_{x\in \magspace} c_x$, and it remains to determine $\tilde{c}_S$. We first locate
the maximal edge conductance and determine its order, by means of classical analysis.
\begin{align}
  \log c_x &= \sum_{y\in \magspace(0,x]} \log \frac{p_y}{q_y} = \sum_{y\in \magspace(0,x]} \log \left(\frac{1-y}{1+y}\cdot \frac{1+\tanh(\beta(y+n^{-1}))}
  {1-\tanh(\beta(y- n^{-1}))}  \right)\nonumber\\
  &= \sum_{y\in \magspace(0,x]} \log \left(\frac{1+g(y)}{1-g(y)}+O(1/n)\right) =
  \sum_{y\in \magspace(0,x]}\log \left(\frac{1+g(y)}{1-g(y)}\right) + O(x)\label{eq-ck-estimate}
\end{align}
%Let $\xi_i$ denote the $i$-th summand above;
Note that $g(x)$ has a unique positive root at $x=\zeta$, and satisfies $g(x) > 0$ for $x\in(0,\zeta)$ and $g(x)
< 0$ for $x> \zeta$. Therefore,
\begin{equation*}
  \log c_x \leq \log c_{\zeta} + O(x) \leq \log c_{\zeta} + O(1) ~,
\end{equation*}
thus we move to estimate $c_{\zeta}$. As $\log c_{\zeta}$ is simply a Riemann sum (after an appropriate
rescaling), we deduce that
\begin{align*}
  \log c_{\zeta} &= \sum_{x\in \magspace(0, \zeta]} \log \left(\frac{1+g(x)}{1-g(x)}\right) + O(1)= \frac{n}{2}\int_{0}^{\zeta} \log \left(\frac{1+g(x)}{1-g(x)}\right)dx + O(1) ~,
\end{align*}
and therefore
\begin{align}
  c_{\zeta} &= \Theta\left( \exp\left(\frac{n}{2}\int_{0}^{\zeta} \log \left(\frac{1+g(x)}{1-g(x)}\right)dx \right)\right)~,\label{eq-c-zeta-bound}\\
  c_x &= O(c_{\zeta})~.\label{eq-ck-c-zeta-bound}
\end{align}
Next, consider the ratio $c_{x+2/n}/c_x$; whenever $x \leq \zeta$, $g(x) \geq 0$, hence we have
\begin{align*}\frac{c_{x+2/n}}{c_x} &= \frac{p_{x+2/n}}{q_{x+2/n}} \geq \frac{1+g(x)}{1-g(x)}-O(1/n) \geq 1 + 2g(x) - O(1/n)~.
\end{align*}
Whenever $ \frac{1}{\sqrt{\temp n}}  \leq  x \leq \zeta - \frac{1}{\sqrt{\temp n}}$ (using the Taylor expansions
around $0$ and around $\zeta$) we obtain that $\tanh(\beta x) - x \geq \frac{1}{2}\sqrt{\temp/n}$. Combining
this with the fact that $x\tanh(\beta x)$ is always non-negative, we obtain that for any such $x$, $2g(x) \geq
\sqrt{\temp/n}$. Therefore, setting
\begin{equation}
  \label{eq-K1-K2-def}
  \xi_1 \deq \sqrt{\frac{1}{\temp n}} ~,~\xi_2 \deq \zeta  - \sqrt{\frac{1}{ \temp n}}~,~
  \xi_3 \deq \zeta  + \sqrt{\frac{1}{\temp n}}~,
\end{equation}
we get
\begin{align}\label{eq-K1-K2-ratio-bound}\frac{c_{x+2/n}}{c_x} \geq 1 + \sqrt{\frac{\temp}{ n}}- O(1/n) ~\mbox{ for any }x \in \magspace [\xi_1, \xi_2]~.
\end{align}
Using the fact that $\temp^2 n \to\infty$, the sum of $c_x$-s in the above range is at most the sum of a
geometric series with a quotient $1/(1+\frac{1}{2}\sqrt{\temp/n})$ and an initial position $c_{\zeta}$:
\begin{align}\label{eq-K1-K2-bound}
  \sum_{x\in \magspace[\xi_1, \xi_2]} c_x \leq   3\sqrt{\frac{n}{\temp}}\cdot c_{\zeta}~.
\end{align}
We now treat $x \geq \xi_3$; since $g(\zeta) = 0$ and $g(x)$ is decreasing for any $x \geq \zeta$, then in
particular whenever $\zeta + \sqrt{\temp / n} \leq x \leq 1$ we have $-1 = g(1) \leq g(x) \leq 0$, and therefore
\begin{align*}\frac{c_{x+2/n}}{c_x} &= \frac{p_{x+2/n}}{q_{x+2/n}} \leq 1+g(x) + O(1/n) ~.
\end{align*}
Furthermore, for any $\zeta + \sqrt{\temp / n} \leq x \leq 1$ (using Taylor expansion around $\zeta$) we have
$\tanh(\beta x) - x \leq - \sqrt{\temp / n}$, and hence $g(x) \leq -\sqrt{\temp / n}$. We deduce that
\begin{align*}
\frac{c_{x+2/n}}{c_x} \leq 1 - \sqrt{\frac{\temp}{n}}+ O(1/n) ~\mbox{ for any }  x \in\magspace[\xi_3, 1]~,
\end{align*}
and therefore
\begin{align}\label{eq-K3-bound}
\sum_{x\in \magspace[\xi_3, 1]} c_x \leq  2\sqrt{\frac{n}{\temp}} \cdot c_{\zeta}~.
\end{align}
Combining \eqref{eq-K1-K2-bound} and \eqref{eq-K3-bound} together, and recalling \eqref{eq-ck-c-zeta-bound}, we
obtain that
\begin{align*}
  \tilde{c}_S = \sum_{x\in\magspace} c_x  \leq 2\left( |\magspace[0,\xi_1]| +5 \sqrt{n/\temp} + |\magspace[\xi_2,\xi_3]|\right) c_{\zeta}=O\left( \sqrt{\frac{n}{\temp}}c_{\zeta}\right)~.
\end{align*}
Finally, consider $x\in \magspace[\xi_2, \xi_3]$; an argument similar to the ones above (i.e., perform Taylor
expansion around $\zeta$ and bound the ratio of $c_{x+2/n}/c_x$) shows that $c_x$ is of order $c_{\zeta}$ in
this region. This implies that for some constant $b>0$
\begin{align}\label{eq-conductance-K2-K3-bound}
\tilde{c}_S \geq \sum_{x\in\magspace[\xi_2, \xi_2]} c_x \geq b |\magspace[\xi_2, \xi_3]|c_{\zeta} \geq
b \sqrt{\frac{n}{\temp}}c_{\zeta}~,
\end{align}
and altogether, plugging in \eqref{eq-c-zeta-bound}, we get
\begin{equation}
  \label{eq-cS-order}
  \tilde{c}_S = \Theta\left(\sqrt{\frac{n}{\temp}}\exp\left(\frac{n}{2}\int_{0}^{\zeta} \log \left(\frac{1+g(x)}{1-g(x)}\right)dx \right)\right)~.
\end{equation}
\end{proof}
\begin{proof}[\textbf{\emph{Proof of Lemma \ref{lem-order-resistance}}}]
Translating the conductances, as given in \eqref{eq-K1-K2-ratio-bound}, to resistances, we get
$$ \frac{r_{x+2/n}}{r_x} \leq 1 - \sqrt{\frac{\temp}{n}}- O(1/n) ~\mbox{ for any }x\in\magspace[\xi_1, \xi_2]~,$$
and hence
$$ \sum_{x\in\magspace[\xi_1, \xi_2]} r_x \leq r_{\xi_1} 2\sqrt{n/\temp} \leq 2\sqrt{n/\temp}~,$$
where in the last inequality we used the fact that $r_{x} \leq r_{x-2/n}~(\leq r_0 = 1)$ for all
$x\in\magspace[0,\zeta]$, which holds since $q_x \leq p_x$ for such $x$. Altogether, we have the following upper
bound:
\begin{align}
  R(0\leftrightarrow \zeta) &= \sum_{x\in\magspace[0,\xi_1]}r_x + \sum_{x\in\magspace[\xi_1, \xi_2]}r_x + \sum_{x\in\magspace[\xi_2, \zeta]}r_x\nonumber\\
&\leq |\magspace[0,\xi_1]| + 2\sqrt{\frac{n}{\temp}} + |\magspace[\xi_2, \zeta]| \leq 4\sqrt{n/\temp}~.
 \label{eq-R-0-zeta-upper-bound}
\end{align}
For a lower bound, consider $x\in\magspace[0,\xi_1]$. Clearly, for any $x \leq \frac{1}{\sqrt{\temp n}}$ we have
$g(x) = \frac{\tanh(\beta x-x)}{1-x\tanh(\beta x)} \leq 2\temp x$, and hence
$$ \frac{r_{x+2/n}}{r_x} = \frac{1-g(x)}{1+g(x)}+O(1/n) \geq 1- 5x \temp \geq \exp(-6x\temp)~,$$
yielding that
$$ r_{\xi_1} \geq \exp\left( -3\temp n \cdot \xi_1^2\right) \geq \mathrm{e}^{-3}~.$$
Altogether,
$$R(0 \leftrightarrow \zeta) \geq \mathrm{e}^{-3} |\magspace[0, \xi_1]|\geq \mathrm{e}^{-4} \sqrt{n/\temp}~,$$
and combining this with \eqref{eq-R-0-zeta-upper-bound} we deduce that $ R(0 \leftrightarrow \zeta) =
\Theta(\sqrt{n/\temp})$.
\end{proof}
\begin{proof}[\textbf{\emph{Proof of Lemma \ref{lem-commute-time}}}]
Plugging in the estimates for $\tilde{c}_S$ and $R(0 \leftrightarrow \zeta)$ in \eqref{eq-commute-time-formula},
we get
\begin{equation}\label{eq-commute-time-0-zeta} \E C_{0,\zeta} = \Theta\left(\frac{n}{\temp}\exp\left(\frac{n}{2}\int_{0}^{\zeta} \log \left(\frac{1+g(x)}{1-g(x)}\right)dx \right)\right)~.\end{equation}
This completes the proof of the lemma \ref{lem-commute-time}.
\end{proof}
Note that by symmetry, the expected hitting time from $\zeta$ to $-\zeta$ is exactly the expected commute time
between $0$ and $\zeta$. Hence,
\begin{align}\label{eq-hitting-zeta--zeta}
\E _{\zeta}[ \tau_{-\zeta} ]= \Theta (\texp)~.
\end{align}

%In order to infer the order of the hitting time from the commute time, we need the following lemmas.
%\begin{lemma}\label{lem-tau-pm-zeta} For this magnetization chain, let $\tau_{\pm\zeta} = \min{t: |S_t| \geq \zeta}$. We have:
%$$\E_{0} \tau_{\pm\zeta} = o (\texp)~.$$
%\end{lemma}
%\begin{proof}
%Observing that for $x \in \magspace[0, \zeta]$, $p_x \geq q_x$ and for $x \in \magspace[-\zeta, 0)$, $q_x \geq
%p_x$, as well as $h_x$ is bounded from above uniformly, we can conclude that $\tau_{\pm\zeta}$ of this
%magnetization chain is stochastically bounded by of $\tau_{\pm\zeta}$ a version of lazy simple random walk in
%the same space with holding probability $h < 1$. Therefore, we obtain that
%$$\E_0 \tau_{\pm\zeta} = O (n^2 \zeta^2)~.$$
%Next, for $\temp = \Omega(1)$, we know that $\texp = \Omega (\exp (c n))$ for some constant $c$, hence $\E_{0}
%\tau_{\pm\zeta} = o (\texp)$; for $\temp = o(1)$, we know that $\zeta = O(\temp)$ and $\texp \geq n/\temp \exp
%(1/2 \temp^2 n)$, which again gives $\E_{0} \tau_{\pm\zeta} = o (\texp)$. This completes the proof.
%\end{proof}
In order to show that the above hitting time is the leading order term in the mixing-time at low temperatures,
we need the following lemma, which addresses the order of the hitting time from 1 to $\zeta$.
\begin{lemma}\label{lem-tau-zeta-from-1}
The normalized magnetization chain $S_t$ in the low temperature regimes satisfies $\E_{1} \tau_{\zeta} = o (\texp)~.$
\end{lemma}
\begin{proof}
First consider the case where $\temp$ is bounded below by some constant. Notice that, as $p_x \leq q_x$ for all $x\geq \zeta$, in this region $S_t$ is a supermartingale. Therefore, Lemma \ref{lem-supermatingale-positive} (or simply standard results on the simple random walk, which dominates our chain in this case) implies that $\E_1 \tau_\zeta = O (n^2)$. Combining this
with the fact that $\texp \geq \exp (c n )$ for some constant $c$ in this case, we immediately obtain that $\E_{1} \tau_{\zeta} =
o (\texp)$.

Next, assume that $\temp = o (1)$. Note that in this case, the Taylor expansion $\tanh(\beta x) = \beta x
- \frac{1}{3}(\beta x)^3 + O((\beta x)^5)$ implies that \begin{align}\label{eq-zeta-approx}\zeta =
\sqrt{3\temp/\beta^3 - O((\beta \zeta))^5} = \sqrt{3\temp} + O(\temp^{3/2})~.\end{align} Recalling that
$\E[S_{t+1} \mid S_t=s] \leq s + \frac{1}{n}(\tanh(\beta s)-s)$ (as $s \geq 0$), Jensen's inequality (using the
concavity of the Hyperbolic tangent) gives
\begin{align}\E[S_{t+1}-S_t] &= \E(\E[ S_{t+1}-S_t\mid S_t])
\leq \frac{1}{n}\left(\E \tanh(\beta S_t)- \E S_t\right)\nonumber\\
&\leq \frac{1}{n}\left(\tanh(\beta \E S_t) - \E S_t\right)~. \label{eq-Zt-from-St=1-bound}
\end{align}
Further note that the function $\tanh(\beta s)$ has the following Taylor expansion around $\zeta$ (for
some $\xi$ between $s$ and $\zeta$):
\begin{align}
\tanh(\beta s) &= \zeta + \beta(1-\zeta^2)(s-\zeta) + \beta^2(-1+\zeta^2) \zeta (s-\zeta)^2\nonumber\\
&+ \frac{\beta^3}{3} (-1+4\zeta^2-\zeta^4)(s-\zeta)^3 + \frac{\tanh^{(4)}(\xi)}{4!}(s-\zeta)^4
~.\label{eq-tanh-taylor-at-zeta}\end{align} Since $\tanh^{(4)}(x) < 5$ for any $x \geq 0$,
\eqref{eq-tanh-taylor-at-zeta} implies that for a sufficiently large $n$ the term $-\frac{1}{3}(s-\zeta)^3$
absorbs the last term in the expansion \eqref{eq-tanh-taylor-at-zeta}. Together with \eqref{eq-zeta-approx}, we
obtain that
\begin{align*}
\tanh(\beta s) &\leq \zeta + \beta(1-\zeta^2)(s-\zeta) + \beta^2(-1+\zeta^2) \sqrt{\temp}
(s-\zeta)^2~.\end{align*} Therefore, \eqref{eq-Zt-from-St=1-bound} follows:
\begin{align}\label{eq-st-drift-low-from-allplus}\E[S_{t+1}-S_t] &\leq- \frac{\sqrt{\temp}}{2n} (\E S_t - \zeta)^2 ~.
\end{align}
Set $$b_i = 2^{-i}~,i_2 = \min\{i:b_i < \sqrt{\temp}\} \mbox{ and }u_i = \min\{t: \E S_t - \zeta < b_i\}~,$$
noting that this gives $b_i/2 \leq \E S_t - \zeta \leq b_i$ for any $t \in [u_i,u_{i+1}]$. It follows that
\begin{align*}
  u_{i+1}-u_i &\leq \frac{b_i / 2}
  {\frac{\sqrt{\temp}}{2n}(\frac{b_i}{2})^2
  } = \frac{4 n}{\sqrt{\temp} b_i}~,
\end{align*}
and hence
\begin{align*}\sum_{i=1}^{i_2} u_{i+1}-u_i &\leq \sum_{i: b_i^2 > \temp} \frac{4 n}{\sqrt{\temp}b_i}=O(n/\temp)~,
\end{align*}
where we used the fact that the series $\{b_i^{-1}\}$ is geometric with ratio $2$. We claim that this implies the required bound on $\E_1 \tau_\zeta$. To see this, recall \eqref{eq-st-drift-low-from-allplus}, according to which $W_t \deq n (S_t-\zeta) \one_{\{ \tau_\zeta > t\}}$ is a supermartingale with bounded increments, whose variance is uniformly bounded from below on the event $\tau_\zeta >t$ (as the holding probabilities of $(S_t)$ are uniformly bounded from above, see \eqref{eq-holding-prob}). Moreover, the above argument gives $\E W_t \leq n\sqrt{\temp}$ for some $t = O(n/\temp)$. Thus, applying
Lemma \ref{lem-supermatingale-positive} and taking expectation, we deduce that $\E_1 \tau_\zeta = O(n/\temp + \temp n^2 ) = O(\temp n^2)$, which in turns gives $\E_1 \tau_\zeta = o (\texp)$.
\end{proof}

\begin{remark*}
With additional effort, we can  establish that $\E_{0} \tau_{\pm\zeta} = o (\texp)$ (for more details, see the
companion paper \cite{DLP-cens}), where $\tau_{\pm\zeta} = \min\{t: |S_t| \geq \zeta\}$. By combining this with
the of $S_t$ symmetry and applying the geometric trial method, we can obtain the expected commute time between $0$ and
$\zeta$: $$\E_\zeta \tau_0 = (\mbox{$\frac{1}{2}$} + o(1)) \E C_{0, \zeta} = \Theta (\texp)~,$$ and therefore
conclude that $\E_{1} \tau_0 = \Theta(\texp)$.
\end{remark*}
%Now, we are ready to establish the order of the hitting time $\tau_0$.
%\begin{lemma}\label{lem-super-tau-0}
%For the high temperature regime, the hitting time $\tau_0$ of the magnetization chain satisfies $\E_1 \tau_0 =
%\Theta (\texp)$.
%\end{lemma}
%\begin{proof}
%First, let's consider $\E_0 \tau_\zeta$. By Lemma \ref{lem-tau-pm-zeta}, we know that the expected hitting time
%to boundaries $\pm \zeta$ is of order $o(\texp)$. By symmetry, with half probability, the magnetization chain
%will hit $\zeta$ at time $\tau_{\pm\zeta}$. On the other half probability of hitting $-\zeta$, we need a trip
%from $-\zeta$ to $0$ and start over. Figuring out this is a geometric trial, we can conclude that $\E_0
%\tau_\zeta = (1+ o(1)) \E_\zeta \tau_0$. Combining this with Lemma \ref{lem-commute-time}, we obtain that
%$\E_\zeta \tau_0 = \Theta(\texp)$. Now, by Lemma \ref{lem-tau-zeta-from-1}, we obtain that $\E_1 \tau_0 = \Theta
%(\texp)$.
%\end{proof}

\subsubsection{Upper bound for mixing}
Combining Lemma \ref{lem-tau-zeta-from-1} and \eqref{eq-hitting-zeta--zeta}, we conclude that $\E_1
\tau_{-\zeta} = \Theta (\texp)$ and hence $\E_1 \tau_0 = O (\texp)$.  Together with Lemma
\ref{lem-taumag-tau-0}, this implies that the magnetization chain will coalescence in $O(\texp)$ steps with
probability arbitrarily close to 1. At this point, Lemma \ref{lem-fullmix-taumag} immediately gives
that the Glauber dynamics achieves full mixing within $O(n\log n)$ additional steps. The following simple
lemma thus completes the proof of the upper bound for the mixing time.
\begin{lemma}
Let $\texp$ be as defined in Lemma \ref{lem-commute-time}. Then $n \log n = o (\texp)$.
\end{lemma}
\begin{proof}
In case $\temp \geq c > 0$ for some constant $c$, we have $\texp \geq n \exp(c' n)$ for some constant $c'>0$ and hence $n \log n = o
(\texp)$. It remains to treat the case $\temp = o(1)$.

Suppose first that $\temp=o(1)$ and $\temp \geq c n^{-1/3}$ for some constant $c>0$. In this case, we have $\texp = \frac{n}{\temp} \exp \left((\frac{3}{4} + o(1))\temp^2 n\right)$ and thus $n \exp (\frac{1}{2} n^{1/3}) = O(\texp)$, giving  $n \log n = o (\texp)$. Finally, if $\temp = o (n^{-1/3})$, we can simply conclude that $n^{4/3} = O(\texp)$ and hence $n \log n = o (\texp)$.
\end{proof}
\subsubsection{Lower bound for mixing}
The lower bound will follow from showing that the probability of hitting $-\zeta$ within $\epsilon \texp$ steps is small, for some small $\epsilon>0$ to be chosen later. To this end, we need the following simple lemma:
\begin{lemma}\label{lem-reverse-markov}
Let $X$ denote a Markov chain over some finite state space $\Omega$, $y\in\Omega$ denote a target state, and $T$
be an integer. Further let $x\in\Omega$ denote the state with the smallest probability of hitting $y$ after at
most $T$ steps, i.e., $x$ minimizes $\P_x(\tau_y \leq T)$. The following holds:
$$\P_x(\tau_y \leq T) \leq \frac{T}{\E_x \tau_y}~.$$
\end{lemma}
\begin{proof}
Set $p=\P_x(\tau_y \leq T)$. By definition, $\P_z(\tau_y \leq T) \geq p$ for all $z\in\Omega$, hence the hitting
time from $x$ to $y$ is stochastically dominated by a geometric random variable with success probability $p$,
multiplied by $T$. That is, we have $\E_x \tau_y \leq T / p$, completing the proof.
\end{proof}

The final fact we would require is that the stationary probability of $\magspace[-1, -\zeta]$ is strictly positive.
This is stated by the following lemma.
\begin{lemma}\label{lem-stationary-zeta-1}
There exists some absolute constant $0 < C_\pi < 1$ such that
$$C_\pi \leq \pi (\magspace[\zeta, 1]) ~(\,=\pi (\magspace[-1, -\zeta])\,)~.$$
\end{lemma}
\begin{proof}
 Repeating the derivation of \eqref{eq-conductance-K2-K3-bound}, we can easily get
$$c_{\magspace[\zeta,1]}\deq \sum_{x\in \magspace[\zeta,1]} (c_x + c'_x) \geq \Theta \left(\sqrt{\frac{n}{\temp}}\exp\left(\frac{n}{2}\int_0^\zeta \log \frac{1+g(x)}{1-g(x)}\right)dx\right) ~.$$
Combining the above bound with \eqref{eq-cS-order}, we conclude that there exists some $C_\pi>0$, such that $\pi
(\magspace[\zeta,1]) \geq C_\pi$.
\end{proof}

Plugging in the target state $-\zeta$ into Lemma \ref{lem-reverse-markov}, and recalling that the
monotone-coupling implies that, for any $T$, the initial state $s_0=1$ has the smallest probability (among all
initial states) of hitting $-\zeta$ within $T$ steps, we deduce that, for a sufficiently small $\epsilon
> 0$,
$$ \P_1( \tau_{-\zeta} \leq \epsilon \texp) \leq \frac{1}{2} C_\pi~.$$
This implies that $$\texp = O\Big(\tmix\Big(\frac{1}{2} C_\pi\Big)\Big)~,$$ which in turn gives $$\texp = O\left(\tmix\big(\mbox{$\frac{1}{4}$}\big)\right)~.$$

\subsection{Spectral gap analysis}

The lower bound is straightforward (as the relaxation time is always at most the mixing time) and we turn to prove the upper bound. Note that, by Lemma
\ref{lem-stationary-zeta-1}, we have $\pi (\magspace[\zeta, 1]) \geq C_\pi >0$. Suppose first that $\gap \cdot
\tmix(\frac{1}{4}) \to \infty$. In this case, one can apply Lemma \ref{lem-hitting-ratio} onto the birth-and-death chain $(S_t)$, with a choice of $\alpha= \pi (\magspace[\zeta,1])$ and $\beta= 1-\pi (\magspace[\zeta,1])$ (recall that $\tmix
(\frac{1}{4})= \Theta ( \E_1 \tau_{-\zeta})$). It follows that
$$\E _{\zeta }\tau_{-\zeta} = o \left(\E_1 \tau_{-\zeta}\right)~.$$
However, as both quantities above should have the same order as $\tmix(\frac{1}{4})$, this leads to a
contradiction. We therefore have $\gap \cdot \tmix (\frac{1}{4})= O (1)$, completing the proof of the upper bound.

\end{document}